\documentclass[11pt]{amsart}
\usepackage[margin=1.15in]{geometry}
\usepackage{amscd,amssymb, amsmath}

\usepackage{graphicx}
\usepackage{amsfonts}
\usepackage{mathrsfs}    
\usepackage{amsmath}    
\usepackage{amsthm}     
\usepackage{amscd}      
\usepackage{amssymb}    
\usepackage{eucal}      
\usepackage{latexsym}   
\usepackage{graphicx}   
\usepackage{verbatim}   
\usepackage[all,cmtip,line]{xy}    
\usepackage{xcolor}
\definecolor{linkcolor}{RGB}{72,162,162}
\usepackage[colorlinks=true,linkcolor=linkcolor,citecolor=linkcolor,
urlcolor=black,bookmarksopen=true]{hyperref}
\usepackage{bookmark}

\usepackage{enumerate}
\usepackage{tikz-cd}
\usepackage{tikz}

 \newlength{\baseunit}               
 \newcount{\numlines}                
\newcounter{IntroThm}                           
\newtheorem{introthm}[IntroThm]{Theorem}        
\newtheorem{introcorollary}[IntroThm]{Corollary}

\newcounter{thmcounter}
\numberwithin{thmcounter}{section}

\numberwithin{equation}{section}

\newtheorem{theorem}[thmcounter]{Theorem}
\newtheorem{proposition}[thmcounter]{Proposition}
\newtheorem{lemma}[thmcounter]{Lemma}

\theoremstyle{definition}
\newtheorem{definition}[thmcounter]{Definition}

\newtheorem{remark}[thmcounter]{Remark}

\newtheorem{question}[thmcounter]{Question}
\newtheorem{conjecture}{Conjecture}
\newtheorem*{conjecture*}{Conjecture}

\newtheoremstyle{claim}{9pt}{3pt}{}{\parindent}{\bf}{.}{1em}{}

\theoremstyle{claim}

\newcommand{\bbZ}{\mathbb{Z}}

\newcommand{\cO}{\sO}

\def\PP{\mathbb{P}}

\newcommand{\sF}{\mathscr{F}}

\newcommand{\sO}{\mathscr{O}}                    

\newcommand{\sI}{\mathscr{I}}                    

\DeclareMathOperator{\gon}{gon}                  

\newcommand{\propernormal}{%
  \mathrel{\ooalign{$\lneq$\cr\raise.22ex\hbox{$\lhd$}\cr}}}

\newcommand{\Sym}{\operatorname{Sym}}

\newcommand{\coker}{\operatorname{coker}}

\newcommand{\lremind}[1]{{\bf[label:  #1]}}
\renewcommand{\lremind}[1]{{}}

\newcommand{\cut}[1]{}

\begin{document}
\pagestyle{plain}
\raggedbottom
\title{Generalizing Green's Syzygy Theorems to Symmetric Products of Curves}
\author[W. Niu]{Wenbo Niu}
\address{Department of Mathematics,University of Arkansas, 850 W Dickson Street, Fayetteville, AR 72701}
\email{\href{mailto:wenboniu@uark.edu}{wenboniu@uark.edu}}
\author[P. Yi Wei]{Peter Yi Wei}
\address{Department of Mathematics,University of Arkansas, 850 W Dickson Street, Fayetteville, AR 72701}
\email{\href{mailto:peterwei.math@gmail.com}{peterwei.math@gmail.com}}

\begin{abstract}
Let $C$ be a smooth projective curve of genus $g$, and let $T_{n,L}$ denote
the tautological line bundle on its $n$-th symmetric product $C_n$. We prove
effective vanishing criteria for Koszul cohomology
$K_{p,q}(C_n,T_{n,B};T_{n,L})$, extending two classical theorems of Mark Green.
The main ingredient is a weight-one vanishing theorem proved by
induction on $n$, using a filtration of the kernel bundle of $T_{n,L}$
to reduce the required cohomological vanishings to lower symmetric products. As a consequence, when $g\geq1$ and $n\geq2$,
$T_{n,L}$ satisfies property $(N_p)$ for
$\deg(L)\geq(n+1)g+p$. This bound is sharp for elliptic symmetric squares
and also gives evidence for a folklore conjecture on adjoint syzygies.
\end{abstract}
\maketitle

\section{Introduction}

\noindent Throughout, we work over an algebraically closed field $\Bbbk$ of
characteristic zero. Let $X$ be a projective variety, let $L$ be a globally
generated ample line bundle, and let $B$ be a line bundle on $X$. Set
$S:=\operatorname{Sym}H^0(X,L)$. The associated module of sections
\[
R_X(B;L):=\bigoplus_{q\in\bbZ}H^0(X,B+qL)
\]
is a finitely generated graded $S$-module and has a minimal graded free resolution over $S$:
\[
\cdots \longrightarrow E_p \longrightarrow \cdots \longrightarrow E_2  \longrightarrow E_1 \longrightarrow E_0 \longrightarrow R_X(B;L) \longrightarrow 0.
\]
Each graded free $S$-module $E_p$ for $p\geq 0$ has the form
\[
E_p=\bigoplus_{q\in \bbZ}K_{p,q}(X,B;L)\otimes S(-p-q).
\]
Here $K_{p,q}(X,B;L)$ is the cohomology of the middle term of
\[
\bigwedge^{p+1}H^0(L)\otimes H^0(B+ (q-1)L)\to\bigwedge^p H^0(L)\otimes H^0(B+ qL) \to\bigwedge^{p-1}H^0(L)\otimes H^0(B+ (q+1)L).
\]
We regard $K_{p,q}(X,B;L)$ as the space of \textit{$p$-th syzygies of weight $q$}. For further background on Koszul cohomology, we refer to
Green's foundational
papers~\cite{Green:KoszulI,Green:KoszulII}.

When $B=\sO_X$, we write $R_X(L):=R_X(\sO_X;L)$ and
$K_{p,q}(X,L):=K_{p,q}(X,\sO_X;L)$. We say $L$
satisfies \emph{property $(N_p)$} if
\[
K_{i,q}(X,L)=0
\quad\text{for all }0\leq i\leq p\text{ and }q\geq2;
\]
equivalently, the resolution of $R_X(L)$ is linear through the $p$-th stage:
\[
\cdots\longrightarrow\bigoplus S(-p-1)\longrightarrow\cdots
\longrightarrow\bigoplus S(-3)\longrightarrow\bigoplus S(-2)
\longrightarrow S\longrightarrow R_X(L)\longrightarrow0.
\]
Classically, property
$(N_0)$ means that $L$ is normally generated, i.e., that the natural map
\[
\Sym^n H^0(X,L)\longrightarrow H^0(X,L^{\otimes n})
\]
is surjective for every $n\geq0$. Property $(N_1)$ means, in addition, that
the homogeneous ideal of $X\subseteq\mathbb{P}(H^0(X,L)^{\vee})$ is generated by
quadrics.

A celebrated theorem of Green~\cite{Green:KoszulI} states that if $C$ is a
smooth projective curve of genus $g\geq0$ and $L$ is a line bundle on $C$
with $\deg(L)\geq2g+1+p$, then $L$ satisfies property $(N_p)$. In
particular,
\[
K_{i,q}(C,L)=0 \quad\text{for all }0\leq i\leq p \textrm{ and } q\geq 2.
\]
Building on earlier work of Castelnuovo, Mattuck--Mayer, Mumford,
Saint-Donat, Fujita, and others, Green's $(2g+1+p)$ theorem helped initiate the modern
study of syzygies on curves and higher-dimensional varieties
~\cite{ein-lazarsfeld-arbitrary,ein-lazarsfeld-asymptotic,Ein:Gonality}.
Green also proved the following vanishing for projective space
\cite{Green:KoszulII}: for $n,d\geq1$, $p\geq0$, and $b\in\bbZ$,
\[
K_{p,q}\big(\PP^n,\sO_{\PP^n}(b);\sO_{\PP^n}(d)\big)=0
\quad\text{whenever}\quad (q-1)d+b\geq p.
\]

These two theorems of Green motivate us to study tautological line bundles on symmetric
products of curves. For $n\geq1$, let $C_n$ denote
the symmetric product of degree $n$, and let
$q_n:C^n\to C_n$ be the quotient map. For a line bundle $L$ on
$C$, the natural $\mathfrak S_n$-equivariant structure on
$L^{\boxtimes n}$ determines a line bundle $T_{n,L}$ on $C_n$
characterized by $q_n^*T_{n,L}=L^{\boxtimes n}$. These bundles also arise
naturally in the geometry of secant varieties of curves \cite{ENP}. The
construction is multiplicative:
$T_{n,L_1+L_2}\cong T_{n,L_1}+ T_{n,L_2}$. It contains both of
Green's settings: $C_1=C$ and
$T_{1,L}=L$, whereas for $C\cong\PP^1$ one has
$C_n\cong\PP^n$ and
$T_{n,\sO_{\PP^1}(d)}\cong\sO_{\PP^n}(d)$.

Our main result extends both vanishings on curves and projective spaces to $C_n$.
\begin{introthm}\label{thm-main}
Let $C$ be a smooth projective curve of genus $g\geq0$, let $B,L$ be line
bundles on $C$, and let $n\geq1$ and $p\geq 0, q\in \bbZ$. Assume
\begin{enumerate}[\normalfont(i)]
\setlength{\itemsep}{0.4\baselineskip}
\item $(q-1)\deg(L)+\deg(B)\geq(n+1)g+p$;
\item $\displaystyle
\deg(L)\ge
\begin{cases}
2g+1, & n\ge2 \text{ or } g=0,\\
2g+p+1, & n=1 \text{ and } g\ge1.
\end{cases}$
\end{enumerate}
Then $K_{p,q}(C_n,T_{n,B};T_{n,L})=0$.
\end{introthm}

\begin{remark}\label{rmk_intro}
When $C\cong\PP^1$, the polarization hypothesis is $\deg(L)\geq1$, as follows
from the genus-zero estimate in the proof of
Lemma~\ref{lem:weight-one-coefficients}. Under the identifications
$C_n\cong\PP^n$ and
$T_{n,\sO_{\PP^1}(d)}\cong\sO_{\PP^n}(d)$, the resulting statement is
precisely Green's vanishing theorem for projective space.
\end{remark}

Taking $B=\sO_C$, Theorem~\ref{thm-main} recovers Green's
$(2g+1+p)$ theorem for property $(N_p)$ when $n=1$.
For higher symmetric products of curves, Theorem~\ref{thm-main} yields
the following property $(N_p)$ statement.
\begin{introcorollary}\label{cor_1}
	When $n\geq2$ and $g\geq1$, the line bundle $T_{n,L}$
	satisfies property $(N_p)$ provided that $\deg(L)\geq (n+1)g+p$.
\end{introcorollary}
\begin{proof}
It is enough to prove that
$K_{i,q}(C_n;T_{n,L})=0$ for $0\leq i\leq p$ and $q\geq2$.
For each such $i$ and $q$, the first numerical hypothesis in
Theorem~\ref{thm-main} follows from
\[
(q-1)\deg(L)\geq\deg(L)\geq(n+1)g+p\geq(n+1)g+i.
\]
If $g\geq1$, $n\geq 2$, then
\[
\deg(L)\geq(n+1)g+p
\geq2g+1.
\]
Thus the second numerical hypothesis also holds.
\end{proof}
\begin{remark}
Suppose that $g=1$ and $n=2$. Then the symmetric square $C_2$ is an
elliptic ruled surface with invariant $e=-1$. Let $[C_0]$ denote the class
of the zero section, where $C_0^2=1$. If $\deg(L)=d$, then
$[T_{2,L}]\equiv d[C_0]$. Combining Corollary~\ref{cor_1} with Park's
necessary criterion~\cite[Theorem~1.1]{Park:RuledSurfacesII}, we conclude
that $T_{2,L}$ satisfies property $(N_p)$ if and only if
$\deg(L)\geq p+3$. Consequently,~\cite[Conjecture~7.3]{GP96} holds for
tautological line bundles on $C_2$, giving a partial confirmation of the
conjecture.
\end{remark}

Corollary~\ref{cor_1} also fits into the broader study of syzygies of adjoint
linear series. Mukai observed that Green's $(2g+1+p)$ theorem can be
reformulated as follows: if $A$ is an ample line bundle on $C$, then
$K_C+mA$ satisfies property $(N_p)$ for every $m\geq p+3$. This perspective
suggests analogous statements in higher dimensions.

For surfaces, Reider's theorem implies that $K_S+mA$ is very ample for
$m\geq4$ when $A$ is ample~\cite{reider-surface}. Motivated by this perspective, Mukai conjectured that $K_S+mA$ should satisfy
property $(N_p)$ for $m\geq p+4$~\cite{ein-lazarsfeld-arbitrary}. In higher dimensions, even Fujita's conjecture remains open in general. It
predicts that $K_X+mA$ is very ample for $m\geq n+2$ on every smooth
projective variety $X$ of dimension $n$. Ein and Lazarsfeld proved that if
$A$ is very ample and $B$ is nef, then $K_X+mA+B$ satisfies property $(N_p)$
for $m\geq n+1+p$~\cite{ein-lazarsfeld-arbitrary}. More recently, Bangere
and Lacini obtained the same bound for $B=\sO_X$ under the weaker assumption
that $A$ is ample and basepoint-free~\cite{bangere-lacini-adjoint}. These
results motivate the following folklore conjecture.
\begin{conjecture}
	Let $X$ be a smooth projective variety of dimension $n$, let $A$ be an
	ample line bundle on $X$, and let $B$ be a nef divisor. Then $K_X+mA+B$
	satisfies property $(N_p)$ for every $m\geq n+2+p$.
\end{conjecture}

Corollary~\ref{cor_1} provides further evidence for this conjecture in the
setting of symmetric products of curves. Recall the finite morphism
$\sigma_n:C_{n-1}\times C\to C_n$, given by
$(\xi,x)\mapsto\xi+x$. Define the divisor $\delta_n$
by $\sO_{C_n}(\delta_n):=
\det\big(\sigma_{n,*}\sO_{C_{n-1}\times C}\big)^*$, where $(\cdot)^*$
denotes dualization. Then the canonical bundle formula is
$K_{C_n}=T_{n,K_C}(-\delta_n)$.

Suppose that $L=K_C+mA+B$, where $A$ is ample. Multiplicativity and the
canonical bundle formula give
\begin{equation*}
	T_{n,L}
	\cong K_{C_n}+ m T_{n,A}+ T_{n,B}(\delta_n).
\end{equation*}
The line bundle $T_{n,A}$ is ample as $A$ is ample. For $g\geq1$,
Proposition~\ref{prop-main} shows that $T_{n,B}(\delta_n)$ is nef if and only
if $\deg(B)\geq(n-1)(g-1)$. If this inequality holds and
$m\geq n+2+p$, then
\[
\deg(L)\geq2g-2+(n+2+p)+(n-1)(g-1)=(n+1)g+p+1.
\]
Thus Corollary~\ref{cor_1} verifies the conjectural adjoint-type vanishing
for these tautological line bundles.

Related questions for the determinant line bundles
$N_{n,L}=T_{n,L}(-\delta_n)$ have been studied by Sheridan.
In his thesis \cite[Chapter~9]{She20}, he formulated a
shifting conjecture for these bundles and explained
its relation to the above conjecture. 
Sheridan subsequently established projective normality results
for determinant line bundles on symmetric squares \cite{She22}.

Theorem~\ref{thm-main} follows from the effective vanishing theorem
for weight-one syzygies below, which is the main technical result of the paper. 

\begin{introthm}\label{thm_weight_one}
Let $C$ be a smooth projective curve of genus $g\geq0$, let $B,L$ be line
bundles on $C$, and let $n\geq1$ and $p\geq0$. Assume
\begin{enumerate}[\normalfont(i)]
\setlength{\itemsep}{0.4\baselineskip}
    \item $\deg(B)\geq(n+1)g+p$;
    \item$\displaystyle
\deg(L)\ge
\begin{cases}
2g+1, & n\ge2 \text{ or } g=0,\\
2g+p+1, & n=1 \text{ and } g\ge1.
\end{cases}$
\end{enumerate}
Then $K_{p,1}\big(C_n,T_{n,B};T_{n,L}\big)=0$.
\end{introthm}
\begin{remark}
Setting $p=0$ in Theorem~\ref{thm_weight_one} gives $K_{0,1}(C_n,T_{n,B};T_{n,L})=0$, equivalently, the natural map
$H^0(C_n,T_{n,L})\otimes H^0(C_n,T_{n,B})
\to H^0(C_n,T_{n,L+B})$ is surjective,
provided that $\deg(L)\geq 2g+1$,
$\deg(B)\ge(n+1)g$.
This recovers Mumford's
multiplication theorem \cite[Theorem~6]{Mum70} when $n=1$: if
$\deg(L)\ge 2g+1$, $\deg(B)\ge 2g$, then
$H^0(C,L)\otimes H^0(C,B)
\to H^0(C,L+B)$ is surjective.
\end{remark}
\begin{remark}
The asymptotic vanishing of weight-one syzygies on smooth varieties of arbitrary dimension are studied by Ein--Lazarsfeld--Yang \cite[Theorem~A]{ELY16}. More precisely, if $B$ is a $p$-jet very ample line bundle on a smooth projective variety $X$, then $K_{p,1}(X,B;dA+P)=0$ for $d\gg 0$, where $A$ is ample and $P$ is arbitrary. While \cite{ELY16} fixes the coefficient bundle $B$ and takes the polarization $L=dA+P$ sufficiently positive, Theorem~\ref{thm_weight_one} proceeds in an orthogonal direction: for $n\geq2$, we fix $\deg(L)\geq 2g+1$ and obtain vanishing by requiring $\deg(B)\geq (n+1)g+p$.
\end{remark}
The proof extends the basepoint-free pencil trick to symmetric products and
adapts the secant-bundle method of~\cite{ein-lazarsfeld-arbitrary}.
Removing general points from $L$ yields a filtration of $M_{T_{n,L}}$.
Its graded pieces reduce the required cohomology vanishings to divisors $X_P\cong C_{n-1}$, providing the induction on $n$.
We introduce admissible tuples to track the possible induction paths through this
filtration.

The paper is organized as follows. Section~\ref{sec-preliminaries} reviews
Koszul cohomology and the geometry of symmetric products and establishes the
nefness criterion for $T_{n,B}(\delta_n)$. In Section~\ref{sec-syzygy}, we
show that Theorem~\ref{thm-main} follows from
Theorem~\ref{thm_weight_one}, study the positivity of $T_{n,L}$, construct a
filtration of its kernel bundle, and use this filtration to prove
Theorem~\ref{thm_weight_one} by induction on the admissible tuples. Finally,
Section~\ref{sec-questions} discusses further questions and directions for
future research.

\medskip

\noindent {\bf Acknowledgments:} We are grateful to Purna Bangere, Lawrence
Ein, Michael Kemeny, Minjae Kwon, Jinhyung Park, and John Sheridan for valuable discussions and
encouragement. The second author thanks Jinhyung Park and KAIST for their
hospitality during his visit in July 2026. Discussions during
this visit inspired several improvements to the paper.
\section{Preliminaries}\label{sec-preliminaries}

\subsection{Koszul cohomology and syzygies}
If $L$ is globally generated, the evaluation morphism
$\operatorname{ev}:H^0(X,L)\otimes\sO_X\rightarrow L$ is surjective. The
\textit{kernel bundle} $M_L$ is defined as its kernel:
\[
0\longrightarrow M_L\longrightarrow H^0(X,L)\otimes\sO_X
\xrightarrow{\mathrm{ev}}L\longrightarrow0.
\]
We will use the kernel bundle $M_L$ frequently to compute Koszul cohomology.
The following lemmas give useful formulas.
\begin{lemma}[{\cite[Proposition 3.2]{ein-lazarsfeld-asymptotic}}]\label{lemma-kernel-bdl}
Let $B$ and $L$ be line bundles on $X$, and suppose that $L$ is globally
generated. Assume that $H^i(X,B+mL)=0$ for all $i>0$ and $m\geq1$. Then,
for every $p\geq0$ and $q\geq2$,
\[
K_{p,q}(X,B;L)=H^1\left(X,\bigwedge^{p+1}M_L
\otimes\sO_X(B+(q-1)L)\right).
\]
If moreover, $H^1(X,\cO_X(B))=0$, then the same statement holds also when $q = 1$.
\end{lemma}

\begin{lemma}[{\cite[Corollary~3.3]{ein-lazarsfeld-asymptotic}}]
\label{cor:el12-asymptotic-syzygies}
Assume the hypotheses of Lemma~\ref{lemma-kernel-bdl}, and set
$n:=\dim X$. Then, for every $p\geq0$ and $q\geq2$,
\[
K_{p,q}(X,B;L)=H^{q-1}\left(X,\bigwedge^{p+q-1}M_L\otimes
\sO_X(B+L)\right).
\]
In particular, $K_{p,n+1}(X,B;L)=H^n(\bigwedge^{p+n}M_L\otimes
\sO_X(B+L))$,
and $K_{p,q}(X,B;L)=0$ for $q>n+1$.
\end{lemma}
We will also use the following form of Green’s vanishing lemma induced from ~\cite[Theorem~3.a.1]{Green:KoszulI}. See also ~\cite[Theorem 5.3.1]{EL26}.
\begin{lemma}
\label{lem:green-vanishing}
Let $L, F$ be line bundles on $X$ and $L$ is globally generated. Then
\[
H^0\left(X,\bigwedge^p M_L\otimes F\right)=0
\]
for every integer $p\geq0$ such that $p\geq h^0(X,F)$.
\end{lemma}

\subsection{Lemmas on vector bundles} We first record how exterior powers behave on an elementary modification of vector bundles along a divisor.

\begin{lemma}\label{lem-wedge-coker}
Let $D$ be an effective integral divisor on a smooth variety $X$. Let $F$ be
a locally free sheaf of rank $n$ on $X$, and let $Q$ be a locally free
$\sO_D$-module of rank one. Suppose that $E$ is an \emph{elementary
transformation} of $F$ along $Q$; equivalently, there is a short exact
sequence
\[
0\longrightarrow E\xrightarrow{\;\,\varphi\;\,}F\longrightarrow i_*Q\longrightarrow0,
\]
where $i:D\hookrightarrow X$ is the inclusion. Let
$M=\ker(F|_D\to Q)$. Then, for every $r\geq1$, there is a short exact
sequence
\[
0\longrightarrow\wedge^r E\xrightarrow{\,\wedge^r\varphi\,}
\wedge^r F
\longrightarrow i_*\left(\wedge^{r-1}M\otimes Q\right)\longrightarrow0.
\]
There is also another exterior-power short exact sequence
\begin{equation}\label{eqn-wedge-coker}
0\longrightarrow (\wedge^r F)(-D)\longrightarrow\wedge^r E
\longrightarrow i_*\wedge^r M\longrightarrow0.
\end{equation}
\end{lemma}
\begin{proof}
We may assume $1\leq r\leq n$, since all the sheaves in question vanish
for $r>n$. Define
\[
    Q_r:=\coker\left(\wedge^r E\xrightarrow{\,\wedge^r\varphi\,}\wedge^r F\right).\]
 The morphism $\varphi$ is an isomorphism away from $D$; so is $\wedge^r\varphi$. Thus $Q_r$ is supported on $D$. Moreover,
\[
\operatorname{Fitt}_0(i_*Q)=\sO_X(-D),
\qquad
\det(E)\otimes\det(F)^{-1}\cong\sO_X(-D).
\]

Restricting the defining elementary transformation to $D$ gives the exterior-power sequence
\begin{equation}\label{eq:wedge-on-D}
0\longrightarrow\wedge^r M
\longrightarrow\wedge^r(F|_D)
\longrightarrow\wedge^{r-1}M\otimes Q
\longrightarrow0.
\end{equation}
The image of $E|_D\to F|_D$ lies in $M$. Consequently, the composite
\[
\wedge^r E\longrightarrow\wedge^r F
\longrightarrow i_*\wedge^r(F|_D)
\longrightarrow i_*\left(\wedge^{r-1}M\otimes Q\right)
\]
vanishes. It therefore induces a surjection
\begin{equation}\label{eq:Qr-to-wedge-M}
f_r:Q_r\longrightarrow
i_*\left(\wedge^{r-1}M\otimes Q\right).
\end{equation}

We show that $f_r$ is an isomorphism. Since $X$ is smooth,
$\wedge^rE$ and $\wedge^rF$ are locally free over regular local
rings. The two-term locally free resolution defining $Q_r$, together with
the fact that its support is the divisor $D$, shows that $Q_r$ is a
Cohen--Macaulay sheaf of pure codimension one. Because $D$ is integral, its
only associated point is the generic point $\eta$ of $D$.

Let $t$ be a uniformizer of the discrete valuation ring $\sO_{X,\eta}$.
The quotient $Q_\eta$ has length one, so Smith normal form gives bases of
$E_\eta$ and $F_\eta$ in which
$\varphi_\eta=\operatorname{diag}(t,1,\ldots,1)$.
If $e_1,\ldots,e_n$ is the resulting basis of $F_\eta$, then the cokernel
of $\wedge^r\varphi_\eta$ is generated over
$\sO_{X,\eta}/(t)$ by the classes
\[
e_1\wedge e_{i_2}\wedge\cdots\wedge e_{i_r},
\qquad 2\leq i_2<\cdots<i_r\leq n.
\]
Thus $(Q_r)_\eta$ is free of rank $\binom{n-1}{r-1}$ over
$\sO_{D,\eta}$. Under~\eqref{eq:Qr-to-wedge-M}, these generators map to
the corresponding basis of
$\wedge^{r-1}M_\eta\otimes Q_\eta$. Hence $(f_r)_\eta$ is an
isomorphism. The kernel of $f_r$ has zero localization at the only
associated point of $Q_r$, and therefore it is zero. This proves the first
exact sequence.

For the second sequence, combine the first one with the restriction
sequence for $\wedge^rF$:
\[
\begin{tikzcd}[row sep=small]
0 \arrow[r] & (\wedge^rF)(-D) \arrow[r] \arrow[d]
& \wedge^rF \arrow[r] \arrow[d,equal]
& i_*\wedge^r(F|_D) \arrow[r] \arrow[d]
& 0 \\
0 \arrow[r] & \wedge^rE \arrow[r]
& \wedge^rF \arrow[r]
& i_*\left(\wedge^{r-1}M\otimes Q\right) \arrow[r]
& 0.
\end{tikzcd}
\]
The right vertical arrow is the quotient map in~\eqref{eq:wedge-on-D}; its
kernel is $i_*\wedge^rM$. The snake lemma gives
\[
0\longrightarrow(\wedge^rF)(-D)
\longrightarrow\wedge^rE
\longrightarrow i_*\wedge^rM
\longrightarrow0,
\]
which is~\eqref{eqn-wedge-coker}.
\end{proof}

We conclude this subsection with the following standard consequence of the
exterior-power filtration associated to a short exact sequence.
\begin{lemma}\label{lem-vanishing-wedge}
Let $0\to E_1\to E_2\to E_3\to0$ be a short exact sequence of vector
bundles, and let $\sF$ be a coherent sheaf on a variety $X$. Fix integers $i,w\geq 0$. Then $H^i(X,\sF\otimes\bigwedge^wE_2)=0$ if
\[
H^i\left(X,\sF\otimes\bigwedge^kE_1\otimes
\bigwedge^{w-k}E_3\right)=0
\qquad(0\leq k\leq w).
\]
\end{lemma}
\begin{proof}
There is a filtration
\[
\bigwedge^wE_2=F^0\supseteq F^1\supseteq\cdots
\supseteq F^w\supseteq F^{w+1}=0,\;\textrm{ with }F^k/F^{k+1}\cong\bigwedge^kE_1\otimes\bigwedge^{w-k}E_3
\]
for $0\leq k\leq w$; see, for example,
\cite[Chapter II, Exercise 5.16(d)]{hartshorne}. Starting with
$F^{w+1}=0$, the long exact sequences of the filtration give the assertion.
\end{proof}
\begin{remark}\label{rmk:01}
More generally, let
$M=F^0\supseteq F^1\supseteq\cdots\supseteq F^r=0$ be a filtration by
subbundles. Iterating the lemma shows that
$H^i(X,\sF\otimes\bigwedge^wM)=0$ provided
\[
H^i\left(X,\sF\otimes
\bigotimes_{\ell=0}^{r-1}
\bigwedge^{a_\ell}(F^\ell/F^{\ell+1})\right)=0
\]
for every $a_0,\ldots,a_{r-1}\geq0$ satisfying
$\sum_{\ell=0}^{r-1}a_\ell=w$.
\end{remark}

\subsection{Symmetric products of curves and tautological line bundles}
Let $C$ be a smooth projective curve of genus $g\geq0$. For $n\geq0$, let
$C_n$ denote its $n$-th symmetric product. We use the conventions
$C_0=\operatorname{Spec}k$ and $T_{0,L}=\sO_{C_0}$. For $n\geq1$, there is
a natural quotient map $q_n:C^n\longrightarrow C_n$. The addition morphism
\[
\sigma_n:C_{n-1}\times C\longrightarrow C_n,
\qquad (\xi,x)\longmapsto \xi+x,
\]
is a finite flat surjective morphism of degree $n$.

\begin{remark}
    We recall the following properties of these line bundles from
    \cite[Definition~3.2 and Remark~3.3]{ENP}:
    \begin{enumerate}[\normalfont(i)]
	    \item \label{rem-T_k-coh} Let $L$ be a line bundle on a curve $C$. For integers $n\geq 1$ and $0\leq i\leq n$, we have
       $$H^i(C_n,T_{n,L})=S^{n-i}\,H^0(C,L)\otimes \wedge^i H^1(C,L);\;$$
        \item $q_n^*\,\sO_{C_n}(\delta_n) = \sO_{C^n}(\Delta_n)$, where $\Delta_{u,v}:=\{(x_1,\cdots,x_{n})\in C^n \,|\,x_u=x_v\}$ is the pairwise diagonal on $C^n$ and $\Delta_n:=\sum_{1\leq u<v\leq n} \Delta_{u,v}$. When $n=1$, we set $\Delta_1=\varnothing$.
        \item For any two line bundles $L_1$ and $L_2$ on $C$, we have $T_{n,L_1}+ T_{n,L_2}=T_{n,L_1+ L_2}$.
        \item Given a point $P\in C$, an effective divisor $X_P$ on $C_n$ is defined to be the image of the morphism $C_{n-1} \to C_n$ sending $\xi$ to $\xi+P$. It is ample and the associated line bundle $\sO_{C_n}(X_P)=T_{n,\sO_C(P)}$. For any line bundle $L$ on $C$, we have $T_{n,L}|_{X_P}=T_{n-1,L}$.
        \item The canonical bundle of $C_n$ is given by $K_{C_n}=T_{n,K_C}(-\delta_n)$.
    \end{enumerate}
\end{remark}
\subsection{Nefness of divisors on symmetric products of curves}
We now establish the sharp nefness criterion for the residual term
$T_{n,B}(\delta_n)$ appearing in the decomposition used in the Introduction.
A related observation appears without proof in~\cite{Ein:Gonality}; here we
give a complete proof.
\begin{proposition}\label{prop-main}
Assume $g\geq1$. Then $T_{n,B}(\delta_n)$ is nef if and only if
$\deg(B)\geq (n-1)(g-1)$.
\end{proposition}

\begin{proof}
    The case $n=1$ is immediate. Assume $n\geq2$ and set
    \[
        D_n:=q_n^*\big(T_{n,B}(\delta_n)\big)
        \cong B^{\boxtimes n}\left(\sum_{1\leq i<j\leq n}\Delta_{i,j}\right).
    \]
    Since $q_n$ is finite and surjective, $T_{n,B}(\delta_n)$ is nef if and
    only if $D_n$ is nef. Let $C_{\delta}\subseteq C^n$ be the small
    diagonal. For $\Delta=\Delta_{1,2}$, adjunction gives
    \[
        \sO_{C^n}(\Delta)|_{\Delta}
        \cong K_C^{\vee}\boxtimes\sO_C^{\boxtimes n-2},
        \qquad
        \sO_{C^n}(\Delta)|_{C_{\delta}}\cong K_C^{\vee}.
    \]
    Hence
    \[
        D_n\cdot C_{\delta}
        =n\,\deg(B)+\frac{n(n-1)}2(2-2g).
    \]
    Nefness therefore requires $\deg(B)\geq(n-1)(g-1)$.

    Conversely, assume this inequality and argue by induction on $n$, the
    case $n=1$ being immediate.
    Let $E\subseteq C^n$ be an irreducible curve, and compute intersections
    after passing to its normalization. Let $F_i\cong C^{n-1}$ be a fiber of
    the projection $p_i:C^{n}\to C$.

    \smallskip
    \noindent\emph{Case 1: $E$ lies in a coordinate fiber.}
    Say $E\subseteq F_1$.
    \[
        D_n|_{F_1}\cong
        B^{\boxtimes(n-1)}\left(\sum_{2\leq i<j\leq n}\overline{\Delta}_{i,j}\right)
        \otimes\sO_{F_1}\left(\sum_{i=2}^n\overline{F}_i\right).
    \]
    Here $\overline{\Delta}_{i,j}$ is the corresponding diagonal in
    $F_1\cong C^{n-1}$ and $\overline{F}_i$ is the coordinate fiber obtained
    by imposing $x_i=x_1$.

    Since $g\geq1$, the assumption
    $\deg(B)\geq(n-1)(g-1)$ implies
    $\deg(B)\geq(n-2)(g-1)$. The first factor is therefore nef by induction,
    while each coordinate fiber is nef. Hence $E\cdot D_n\geq0$.

    \smallskip
    \noindent\emph{Case 2: $E$ lies in no coordinate fiber.}
    Set $d_i:=\deg(p_i|_E)\geq1$. Then
    $E\cdot p_i^*B=\deg(B)d_i$. If $E\not\subseteq\Delta_{i,j}$, then
    $E\cdot\Delta_{i,j}\geq0$. If $E\subseteq\Delta_{i,j}$, then
    $d_i=d_j=:d_{i,j}$, and adjunction gives
    \[
        E\cdot\Delta_{i,j}=d_{i,j}(2-2g).
    \]
    Writing $S=\{(i,j)\mid E\subseteq\Delta_{i,j}, 1\leq i<j\leq n\}$, we obtain
    \begin{align*}
        E\cdot D_n
        &\geq \deg(B)\sum_{1\leq i\leq n} d_i +(2-2g)\sum_{(i,j)\in S}d_{i,j}
		\geq \deg(B)\sum_{i=1}^n d_i
        +(1-g)\sum_{1\leq i<j\leq n}(d_i+d_j)\\
        &\geq (g-1)\left((n-1)\sum_{i=1}^n d_i
        -\sum_{1\leq i<j\leq n}(d_i+d_j)\right)=0,
    \end{align*}
    where the first inequality holds because $E$ has nonnegative
    intersection with every pairwise diagonal that does not contain it. The
    second uses $g\geq1$ and $d_i+d_j=2d_{i,j}$ for $(i,j)\in S$, while the
    last uses $\deg(B)\geq (n-1)(g-1)$ together with
    \[
    (n-1)\sum_{i=1}^n d_i
    =\sum_{1\leq i<j\leq n}(d_i+d_j).
    \]
    Therefore, $D_n$, and hence $T_{n,B}(\delta_n)$, is nef.
\end{proof}

\section{Syzygies of tautological line bundles}\label{sec-syzygy}

We first show that Theorem~\ref{thm-main} follows formally from the
weight-one vanishing in Theorem~\ref{thm_weight_one}.

\begin{proof}[Proof of Theorem~\ref{thm-main} assuming Theorem~\ref{thm_weight_one}]
Set $B_q:=B+(q-1)L$. By the definition of Koszul cohomology, shifting the coefficient line bundle
by $(q-1)L$ identifies
\[
K_{p,q}\big(C_n,T_{n,B};T_{n,L}\big)
=K_{p,1}\big(C_n,T_{n,B_q};T_{n,L}\big).
\]
The first hypothesis of Theorem~\ref{thm-main} says precisely that
\[
\deg(B_q)=(q-1)\deg(L)+\deg(B)\geq(n+1)g+p.
\]
The second hypothesis is exactly the corresponding hypothesis on $L$ in
Theorem~\ref{thm_weight_one}. Applying Theorem~\ref{thm_weight_one} yields
$K_{p,1}\big(C_n,T_{n,B_q};T_{n,L}\big)=0$, and hence
$K_{p,q}\big(C_n,T_{n,B};T_{n,L}\big)=0$.
\end{proof}

It remains to prove Theorem~\ref{thm_weight_one}, which establishes the
effective vanishing of weight-one syzygies for tautological bundles on
$C_n$. We begin by recording basic positivity and vanishing results.

\subsection{\texorpdfstring{$p$}{p}-very ampleness}
\begin{lemma}\label{lem-vanishing-T}
    Let $L$ be a line bundle on a smooth projective curve $C$ of genus
    $g\geq0$. If $\deg(L)\geq \max\{0,2g-1\}$, then
    $H^i(C_n,mT_{n,L})=0$ for all $i,m>0$.
\end{lemma}
\begin{proof}
Since $m\geq1$,
$\deg(mL)\geq\max\{0,2g-1\}$, and hence $H^1(C,mL)=0$. For $1\leq i\leq n$,
\[
H^i(C_n,T_{n,mL})
=S^{n-i}H^0(C,mL)\otimes \wedge^i H^1(C,mL)=0.
\]
The assertion for $i>n$ follows from $\dim C_n=n$.
\end{proof}
\begin{definition}
Let $X$ be a projective variety and let $p\geq0$. A line bundle $L$ on $X$
is \emph{$p$-very ample} if, for every zero-dimensional subscheme
$\xi\subseteq X$ of length $p+1$, the restriction map
$H^0(X,L)\longrightarrow H^0(\xi,L|_{\xi})$
is surjective.
\end{definition}
\begin{proposition}\label{prop-p-very-ample}
Let $L$ be a $p$-very ample line bundle on a smooth projective curve $C$,
where $p\geq0$. Then $T_{n,L}$ is $p$-very ample on $C_n$ for every
$n\geq1$.
\end{proposition}
\begin{proof}
We use induction on $n$, with a secondary induction on $p$. The case $n=1$
is immediate because $T_{1,L}=L$. Let $n\geq2$, and let $z\subseteq C_n$
have length $p+1$. We must prove that
\[
r_{z,n,L}:H^0(C_n,T_{n,L})\longrightarrow H^0(z,T_{n,L}|_z)
\]
is surjective.

Choose $P\in C$ such that $X_P$ meets the support of $z$. Let
$y=z\cap X_P$ be the scheme-theoretic intersection and let $x$ be the
residual subscheme, defined by $\sI_x=(\sI_z:\sI_{X_P})$. The residual
sequence, together with
$T_{n,L}(-X_P)\cong T_{n,L(-P)}$ and
$T_{n,L}|_{X_P}\cong T_{n-1,L}$, gives the commutative diagram
\[
\begin{tikzcd}
    0 \arrow[r] & H^0(  T_{n,L}(-X_P)) \arrow[r] \arrow[d, "r_{x,n,L(-P)}"] & H^0( T_{n,L}) \arrow[r] \arrow[d, "r_{z,n,L}"] & H^0(  T_{n,L}|_{X_P}) \arrow[r] \arrow[d, "r_{y,n-1,L}"] & 0 \\
    0 \arrow[r] & H^0(  T_{n,L}(-X_P)|_x) \arrow[r] & H^0( T_{n,L}|_z) \arrow[r] & H^0( T_{n-1,L}|_y) \arrow[r] & 0
\end{tikzcd}
\]
The top row is exact because it identifies with
\begin{equation}\label{eqn_symmetric_restriction}
0\longrightarrow S^n H^0(L(-P)) \longrightarrow S^nH^0(L) \longrightarrow S^{n-1} H^0(L) \longrightarrow 0,
\end{equation}
where the last nonzero map is contraction by evaluation at $P$ in characteristic zero.

Since $\operatorname{length}(y)\leq p+1$, induction on $n$ shows that
$r_{y,n-1,L}$ is surjective. Then the surjectivity of $r_{z,n,L}$ follows from surjectivity of $r_{x,n,L(-P)}$. If $x=\varnothing$, it
is trivial. Otherwise, by our choice of $X_P$, $y\neq\varnothing$, so
$\operatorname{length}(x)\leq p$. Moreover, the $p$-very ampleness of $L$
implies that $L(-P)$ is $(p-1)$-very ample. Induction on $p$ shows that $r_{x,n,L(-P)}$ is surjective.
\end{proof}

\subsection{Filtration of the kernel bundle \texorpdfstring{$M_{T_{n,L}}$}{MT(n,L)}}\label{sec-filtration}
If $L$ is globally generated on $C$, then Proposition~\ref{prop-p-very-ample}
shows that $T_{n,L}$ is globally generated on $C_n$. Its kernel bundle
$M_{T_{n,L}}$ is defined as follows
$$0\longrightarrow M_{T_{n,L}}\longrightarrow S^nH^0(L)\otimes \sO_{C_n}\longrightarrow T_{n,L}\longrightarrow 0.$$
Let $P\in C$. It determines an effective divisor $X_P\cong C_{n-1}$, the
image of the embedding $C_{n-1}\hookrightarrow C_n$ sending $\xi$ to
$\xi+P$. The following lemma describes the restriction of $M_{T_{n,L}}$ to
$X_P$.
\begin{lemma}\label{lm:01} Under the above assumption, one has a short exact sequence on $X_P$:
	$$0\longrightarrow S^nH^0(L(-P))\otimes\sO_{X_P}\longrightarrow M_{T_{n,L}}|_{X_P}\longrightarrow M_{T_{n-1,L}}\longrightarrow 0$$
\end{lemma}
\begin{proof} As $L$ is globally generated, $T_{n,L}$ and $T_{n-1,L}$ are globally generated. Restricting the defining sequence of $M_{T_{n,L}}$ to $X_P$ gives rise to the following commutative diagram
	\[
	\begin{tikzcd}
		& 0 \arrow[d]           & 0 \arrow[d]           &           &  \\
		     & S^nH^0(L(-P))\otimes \sO_{X_P} \arrow[r,equal ] \arrow[d] &  S^nH^0(L(-P))\otimes \sO_{X_P} \ar[d]  &    &   \\
		0 \arrow[r] & M_{T_{n,L}}|_{X_P} \arrow[r] \arrow[d] & S^nH^0(L)\otimes \sO_{X_P}  \arrow[r] \arrow[d] & T_{n,L}|_{X_P}  \arrow[r] \arrow[d,equal] & 0 \\
		0 \arrow[r] & M_{T_{n-1,L}}\arrow[r] \arrow[d] & S^{n-1}H^0(L)\otimes \sO_{X_P}  \arrow[r] \arrow[d] & T_{n-1,L}\arrow[r]  & 0 \\
		& 0                     & 0                     &                     &
	\end{tikzcd}
	\]
The left-hand side column gives the desired exact sequence.
\end{proof}

 Recall that $\sO_{C_n}(-X_P)\cong T_{n,\sO_C(-P)}$ and
$T_{n,L}|_{X_P}\cong T_{n-1,L}$. Regarding $T_{n-1,L}$ as a sheaf on
$C_n$ via the inclusion $i:X_P\hookrightarrow C_n$, its evaluation map is
surjective. Define $\Gamma^{(n)}_{L,P}$ by the exact sequence
 \begin{equation}\label{eqn-gamma-defn}
		0\longrightarrow \Gamma^{(n)}_{L,P} \longrightarrow S^{n-1} H^0(L)\otimes \sO_{C_n} \longrightarrow i_* T_{n-1,L} \longrightarrow 0.
 \end{equation}
Restricting $T_{n,L}$ to $X_P$ gives an exact sequence on $C_n$:
\begin{equation}
0\longrightarrow T_{n,L(-P)}\longrightarrow T_{n,L}
\longrightarrow i_*T_{n-1,L}\longrightarrow0.
\end{equation}
The corresponding kernel bundles are related by $\Gamma_{L,P}^{(n)}$ as follows.
\begin{lemma}\label{lm:02}
Assume that $L$ and $L(-P)$ are globally generated. Then we have a short exact sequence
\[
0\longrightarrow M_{T_{n,L(-P)}}\longrightarrow M_{T_{n,L}}
\longrightarrow\Gamma^{(n)}_{L,P}\longrightarrow0
\]
\end{lemma}
\begin{proof}
The evaluation sequences and the restriction sequence fit into the
commutative diagram
\[
\begin{tikzcd}[column sep=small, row sep=small]
            & 0 \arrow[d]           & 0 \arrow[d]           & 0 \arrow[d]           &  \\
0 \arrow[r] & M_{T_{n,L(-P)}} \arrow[r] \arrow[d] &  M_{T_{n,L}} \arrow[r] \arrow[d] & \Gamma^{(n)}_{L,P}  \arrow[r] \arrow[d] & 0  \\
0 \arrow[r] & S^nH^0(L(-P))\otimes \sO_{C_n} \arrow[r] \arrow[d] & S^nH^0(L)\otimes \sO_{C_n}  \arrow[r] \arrow[d] & S^{n-1} H^0(L)\otimes \sO_{C_n}  \arrow[r] \arrow[d] & 0 \\
0 \arrow[r] & T_{n,L(-P)} \arrow[r] \arrow[d] & T_{n,L} \arrow[r] \arrow[d] & T_{n-1,L}\arrow[r] \arrow[d] & 0 \\
            & 0                     & 0                     & 0                     &
\end{tikzcd}
\]
The snake lemma gives the exact top row.
\end{proof}

We next restrict $\Gamma_{L,P}^{(n)}$ to another divisor $X_{P'}$.
\begin{lemma}
Assume $n\geq2$ and that $L$ and $L(-P)$ are globally generated. If
$P'=P$, then there is an exact sequence
\begin{equation}\label{eqn-diagram-gamma-rest-factor}
0\longrightarrow T_{n-1,L(-P)}
\longrightarrow\Gamma^{(n)}_{L,P}|_{X_P}
\longrightarrow M_{T_{n-1,L}}
\longrightarrow0,
\end{equation}
whereas, if $P'\neq P$, there is an exact sequence
\begin{equation}\label{eqn-gamma-rest-middle}
0\longrightarrow S^{n-1}H^0(L(-P'))\otimes\sO_{C_{n-1}}
\longrightarrow\Gamma_{L,P}^{(n)}|_{X_{P'}}
\longrightarrow\Gamma_{L,P}^{(n-1)}
\longrightarrow0.
\end{equation}
\end{lemma}
\begin{proof}
Restricting (\ref{eqn-gamma-defn}) to $X_{P'}$ gives two different results, according to whether $P'=P$ or $P'\neq P$. The distinction arises from the restriction of $T_{n-1,L}$ to $X_{P'}$:
\[
T_{n-1,L}\otimes_{\sO_{C_n}}\sO_{X_{P'}}
\cong
\begin{cases}
T_{n-1,L},&P'=P,\\
T_{n-2,L},&P'\neq P.
\end{cases}
\]

If $P'=P$, the derived restriction gives
\begin{equation*}
0\longrightarrow
\operatorname{Tor}_1^{\sO_{C_n}}(T_{n-1,L},\sO_{X_P})
\longrightarrow\Gamma^{(n)}_{L,P}|_{X_P}
\longrightarrow S^{n-1}H^0(L)\otimes\sO_{C_{n-1}}
\longrightarrow T_{n-1,L}\longrightarrow0.
\end{equation*}
The identification
$\operatorname{Tor}_1^{\sO_{C_n}}(T_{n-1,L},\sO_{X_P})
\cong T_{n-1,L}\otimes\sO_{X_P}(-X_P)
\cong T_{n-1,L(-P)}$
shows that the image of
$\Gamma^{(n)}_{L,P}|_{X_P}$ in
$S^{n-1}H^0(L)\otimes\sO_{C_{n-1}}$ is precisely
$M_{T_{n-1,L}}$, proving \eqref{eqn-diagram-gamma-rest-factor}.

If $P'\neq P$, the divisors $X_P$ and $X_{P'}$ have no common component,
so there is no $\operatorname{Tor}_1$ term. Comparing the restricted
sequence with the definition of $\Gamma_{L,P}^{(n-1)}$ gives the commutative
diagram
\[
		\begin{tikzcd}[column sep=small, row sep=small]
			& 0 \arrow[d]           & 0 \arrow[d]           &           &  \\
			     & S^{n-1}H^0(L(-P'))\otimes \sO_{C_{n-1}} \arrow[r,equal ] \arrow[d] &  S^{n-1}H^0(L(-P'))\otimes \sO_{C_{n-1}} \ar[d]  &    &   \\
			0 \arrow[r] & \Gamma_{L,P}^{(n)}|_{X_{P'}} \arrow[r] \arrow[d] & S^{n-1}H^0(L)\otimes \sO_{C_{n-1}}  \arrow[r] \arrow[d] & T_{n-2,L}  \arrow[r] \arrow[d,equal] & 0 \\
			0 \arrow[r] & \Gamma_{L,P}^{(n-1)}\arrow[r] \arrow[d] & S^{n-2}H^0(L)\otimes \sO_{C_{n-1}}  \arrow[r] \arrow[d] & T_{n-2,L}\arrow[r]  & 0 \\
			& 0                     & 0                     &                     &
		\end{tikzcd}
\]
The middle column is \eqref{eqn_symmetric_restriction} with $n$ replaced by
$n-1$. The left column gives \eqref{eqn-gamma-rest-middle}.
\end{proof}

Applying Lemma~\ref{lem-wedge-coker} to \eqref{eqn-gamma-defn} gives the
exterior-power sequence used throughout the proof:
\begin{equation}\label{eqn-wedge-gamma-middle}
0\longrightarrow
\wedge^aS^{n-1}H^0(L)\otimes T_{n,\sO_C(-P)}
\longrightarrow
\wedge^a\Gamma^{(n)}_{L,P}
\longrightarrow
\wedge^aM_{T_{n-1,L}}
\longrightarrow0,
\end{equation}
for $1\leq a\leq h^0(C_{n-1},T_{n-1,L})$, where the last term is viewed
as a sheaf on $X_P$.

\noindent\emph{Set-up.} Assume now that $\deg L\geq2g+1$ and write $r:=h^0(L)-1$. Thus
$h^1(C,L)=0$, $r=\deg L-g$, and $r-1\geq g$. We may choose points
$P_0,\ldots,P_{r-2}$ sufficiently general, so that, after setting
\[
L_0:=L,
\quad
L_{\ell}:=L(-P_0-\cdots-P_{\ell-1}),
\quad(1\leq\ell\leq r-1),\quad A:=L_{r-1}=L(-P_0-\cdots-P_{r-2})
\]
every $L_\ell$ is nonspecial and globally generated and $A$
is a basepoint-free pencil of degree $g+1$. For $g=0$, one has $A\cong\sO_{\PP^1}(1)$ directly.
Accordingly, all the kernel
bundles below are defined. We later refer to $\{P_{\ell}\}_{\ell=0}^{r-2}$ as \textit{flag points}. By Lemma~\ref{lm:02}, we obtain vector bundles
$$\Gamma^{(n)}_{L_{\ell}}:=\Gamma^{(n)}_{L_{\ell},P_{\ell}}\cong M_{T_{n,L_{\ell}}}/M_{T_{n,L_{\ell+1}}} \text{ for }0\leq \ell\leq r-2,$$
and by (\ref{eqn-gamma-defn}), each $\Gamma^{(n)}_{L_\ell}$ sits in the following exact sequence:
$$0 \longrightarrow \Gamma^{(n)}_{L_{\ell}} \longrightarrow S^{n-1}H^0(L_{\ell})\otimes \sO_{C_n} \longrightarrow T_{n-1,L_{\ell}}\longrightarrow 0.$$
Thus we obtain a filtration of $M_{T_{n,L}}$,
\begin{equation}\label{filtration}
	M_{T_{n,L}}=F^0\supseteq F^1 \supseteq \cdots \supseteq F^{r-1}\supseteq F^{r}=0
\end{equation}
in which $F^{\ell}=M_{T_{n,L_{\ell}}}$ for $0\leq \ell\leq r-1$,  $F^{\ell}/F^{\ell+1}\cong \Gamma^{(n)}_{L_{\ell}}$ for $ 0\leq \ell\leq r-2$,  and $F^{r-1}/F^{r}\cong M_{T_{n,A}}$.
\subsection{Proof of Theorem~\ref{thm_weight_one}}
Throughout the remainder of this section, we fix $C,L,B,g,n,p$,
and assume the following two conditions:
\begin{enumerate}[\normalfont(1)]
\setlength{\itemsep}{0.4\baselineskip}
    \renewcommand{\theenumi}{H\arabic{enumi}}
    \item \label{hyp:weight-one-B}$\deg(B)\geq(n+1)g+p$;
    \item \label{hyp:weight-one-L} $\displaystyle
\deg(L)\ge
\begin{cases}
2g+1, & n\ge2 \text{ or } g=0,\\
2g+p+1, & n=1 \text{ and } g\ge1.
\end{cases}$
\end{enumerate}

Condition~(\ref{hyp:weight-one-L}) makes
$L$ globally generated, while condition~(\ref{hyp:weight-one-B}) combined with Lemma ~\ref{lem-vanishing-T} gives $H^i(C_n,T_{n,B}+mT_{n,L})=0$ for all $i>0,m\geq 0$. Thus by Lemma~\ref{lemma-kernel-bdl}, the assertion reduces to
\begin{equation}\label{eq:weight-one-target}
H^1\left(
C_n,\bigwedge^{p+1}M_{T_{n,L}}\otimes T_{n,B}
\right)=0.
\end{equation}

Choose the filtration in \eqref{filtration} from general points
$\{P_{\ell}\}_{\ell=0}^{r-2}$. Remark~\ref{rmk:01}
reduces \eqref{eq:weight-one-target} to the following vanishing (\ref{eq:01}) for graded terms indexed by $a_0,\ldots,a_{r-2}, a\geq0$,
\begin{equation}\label{eq:01}
	H^1\left(
	C_n,
	\bigotimes_{\ell=0}^{r-2}
	\bigwedge^{a_\ell}\Gamma^{(n)}_{L_\ell}
	\otimes\bigwedge^aM_{T_{n,A}}
	\otimes T_{n,B}
	\right)=0,\qquad
	a+\sum_{\ell=0}^{r-2}a_\ell=p+1.
\end{equation}
If $a=0$, the above
vanishing is
\begin{equation*}
H^1\left(
C_n,
\bigotimes_{\ell=0}^{r-2}
\bigwedge^{a_\ell}\Gamma^{(n)}_{L_\ell}
\otimes T_{n,B}
\right)=0,
\quad\textrm{with }
\sum_{\ell=0}^{r-2}a_\ell=p+1.
\tag{$\mathrm A$}
\end{equation*}
If $a\geq n+1$, as $\operatorname{rk}M_{T_{n,A}}=n$ the vanishing (\ref{eq:01}) holds automatically. For $1\leq a\leq n$, consider the Koszul resolution of
$\bigwedge^aM_{T_{n,A}}$:
\begin{equation}\label{eq:weight-one-koszul-resolution}
\begin{split}
0\longrightarrow
&\bigwedge^{n+1}S^nH^0(A)
   \otimes T_{n,-(n+1-a)A}
\longrightarrow\cdots
\longrightarrow
\bigwedge^{a+j}S^nH^0(A)\otimes T_{n,-jA}
\longrightarrow\cdots \\
& \longrightarrow
\bigwedge^{a+1}S^nH^0(A)\otimes T_{n,-A}
\longrightarrow
\bigwedge^aM_{T_{n,A}}
\longrightarrow0.
\end{split}
\end{equation}
After tensoring this resolution by the remaining factors, the associated
long exact cohomology sequences show that the required vanishing (\ref{eq:01}) follows from
\begin{equation*}
H^i\left(
C_n,
\bigotimes_{\ell=0}^{r-2}
\bigwedge^{a_\ell}\Gamma^{(n)}_{L_\ell}
\otimes T_{n,B-iA}
\right)=0,
\quad
1\leq i\leq n+1-a.
\tag{$\mathrm B$}
\end{equation*}

Thus the vanishings $(\mathrm A)$ and $(\mathrm B)$ suffice for
\eqref{eq:weight-one-target}. We successively apply
\eqref{eqn-wedge-gamma-middle} to the nonzero $\wedge^{a_{\ell}}\Gamma_{L_{\ell}}$-factors to reduce $(\mathrm A)$ and $(\mathrm B)$  to the vanishing involving fewer $\Gamma$-factors and on a symmetric product of smaller dimension. Here we illustrate the essential idea by removing one such factor.  Assume that $\bar{\ell}$ is the maximal index such that $a_{\bar{\ell}}$ is  nonzero. We will remove $\Gamma^{(n)}_{L_{\bar{\ell}}}$. Identifying $C_{n-1}\cong X_{P_{\bar{\ell}}}$ and using \eqref{eqn-wedge-gamma-middle}, the vanishing in $(\mathrm A)$ and $(\mathrm B)$  will be obtained by the vanishing of the following groups
\begin{equation*}
	H^1\left(
	C_n,
	\bigotimes_{{\ell<{\bar{\ell}}}}
	\bigwedge^{a_\ell}\Gamma^{(n)}_{L_\ell}
	\otimes T_{n,B-P_{\bar{\ell}}}
	\right), 	\qquad 
	H^1\left(
	C_{n-1},
	\bigotimes_{{\ell<{\bar{\ell}}}}
	\bigwedge^{a_\ell}\Gamma^{(n)}_{L_\ell}|_{C_{n-1}}
	\otimes T_{n-1,B}\otimes \wedge^{a_{\bar{\ell}}}M_{T_{n-1,L_{\bar{\ell}}}}
	\right), 
\end{equation*}
$$
H^i\left(
C_n,
\bigotimes_{\ell<\bar{\ell}}
\bigwedge^{a_\ell}\Gamma^{(n)}_{L_\ell}
\otimes T_{n,B-iA-P_{\bar{\ell}}}
\right), \qquad
H^i\left(
C_{n-1},
\bigotimes_{\ell<\bar{\ell}}
\bigwedge^{a_\ell}\Gamma^{(n)}_{L_\ell}|_{C_{n-1}}
\otimes T_{n-1,B-iA}\otimes \wedge^{a_{\bar{\ell}}}M_{T_{n-1,L_{\bar{\ell}}}}
\right).
$$
Next we need to deal with the bundles $\bigwedge^{a_\ell}\Gamma^{(n)}_{L_\ell}|_{C_{n-1}}$ with $\ell<\bar\ell$ and the newly introduced bundle $\wedge^{a_{\bar{\ell}}}M_{T_{n-1,L_{\bar{\ell}}}}$. For one such $\Gamma_{L_\ell}$-factor, as $P_{\ell}\neq P_{\bar\ell}$, equation~\eqref{eqn-gamma-rest-middle} gives
\[
0\longrightarrow
\textrm{trivial bundle}
\longrightarrow
\Gamma^{(n)}_{L_\ell}\big|_{C_{n-1}}
\longrightarrow
\Gamma^{(n-1)}_{L_\ell}
\longrightarrow0.
\]
Thus all the above vanishings can be reduced to vanishings involving $\wedge^{a'_\ell}\Gamma^{(n-1)}_{L_\ell}$ with $a'_{\ell}\leq a_{\ell}$. To deal with $\wedge^{a_{\bar{\ell}}}M_{n-1,L_{\bar{\ell}}}$, we use the same set of general points $P_{\bar{\ell}},\cdots P_{r-2}$ to get a filtration for $M_{T_{n-1,L_{\bar{\ell}}}}$, and then repeat the filtration construction on $C_{n-1}$ using $\Gamma^{(n-1)}$-bundles.

In summary, at each
step, the kernel term remains on the same symmetric product and subtracts
one point $P$ from the coefficient of the divisors defined in the $T$-line bundle, whereas the quotient term is supported
on a divisor $X_P$, a symmetric product of dimension one less. We therefore argue by induction on
$m=\dim C_m$; only the quotient terms lower the dimension.

The two types of terms require different numerical data. Kernel terms
accumulate point twists at the same level $m$. Quotient terms pass from
$C_m$ to $C_{m-1}$, and the resulting decrease in exterior degree offsets
the accumulated point twists. In addition, resolving an exterior power of
$M_{T_{m-1,A}}$ may introduce one further twist by $A$. To record all the possible
coefficient changes that persist under dimension reduction, set
\[
B_{\boldsymbol c,k}:=B-D_{\boldsymbol c}-kA,
\qquad
D_{\boldsymbol c}:=\sum_{\ell=0}^{r-2}c_\ell P_\ell,
\qquad
|\boldsymbol c|:=\sum_{\ell=0}^{r-2}c_\ell.
\]

For integers $m,p_0,k$ and
$\boldsymbol c=(c_0,\ldots,c_{r-2})\in\mathbb Z^{r-1}$, we call the tuple
$(m,p_0,\boldsymbol c,k)$ \textit{admissible} if
\begin{equation*}
1\leq m\leq n,\qquad p_0\geq0,\qquad
0\leq k,c_\ell\leq n-m,\qquad
|\boldsymbol c|+k+p_0\leq p.
\end{equation*}
Here $m$ is the dimension of the current symmetric product, $p_0+1$ is the
total exterior degree, and $k$ records the number of persistent twists by
$A$. After $n-m$ dimension reductions, each $c_\ell$ and $k$ can increase by
at most one at each step. The final inequality records the corresponding
decrease in exterior degree. Put $e:=\deg(B)-(n+1)g-p\geq0$.
Condition~(\ref{hyp:weight-one-B}) controls all these losses, since
\begin{align}
\deg(B_{\boldsymbol c,k})
&=\deg(B)-|\boldsymbol c|-k(g+1)\notag\\
&\geq(n+1-k)g+p_0+e \geq(m+1)g+p_0.
\label{eq:weight-one-B-degree}
\end{align}

For admissible data, let $(\mathbf I_m)$ denote the assertion that
\begin{equation*}
H^1\left(
C_m,
\bigotimes_{\ell=0}^{r-2}
\bigwedge^{a_\ell}\Gamma^{(m)}_{L_\ell}
\otimes T_{m,B_{\boldsymbol c,k}}
\right)=0
\quad\text{whenever}\quad
\sum_{\ell=0}^{r-2}a_\ell=p_0+1.
\tag{$\mathbf I_m$}
\end{equation*}
For $1\leq a\leq m$, let $(\mathbf {II}_m)$ denote the assertion that
\begin{equation*}
H^1\left(
C_m,
\bigotimes_{\ell=0}^{r-2}
\bigwedge^{a_\ell}\Gamma^{(m)}_{L_\ell}
\otimes\bigwedge^aM_{T_{m,A}}
\otimes T_{m,B_{\boldsymbol c,k}}
\right)=0
\quad\text{whenever}\quad
a+\sum_{\ell=0}^{r-2}a_\ell=p_0+1.
\tag{$\mathbf {II}_m$}
\end{equation*}
For $1\leq a\leq m$, let $(\mathbf {III}_m)$ denote
the family of vanishings for $ 1\leq i\leq m+1-a$:
\begin{equation*}
H^i\left(
C_m,
\bigotimes_{\ell=0}^{r-2}
\bigwedge^{a_\ell}\Gamma^{(m)}_{L_\ell}
\otimes T_{m,B_{\boldsymbol c,k}-iA}
\right)=0
\quad\textrm{whenever}\quad
a+\sum_{\ell=0}^{r-2}a_\ell=p_0+1.
\tag{$\mathbf {III}_m$}
\end{equation*}
Here $a$ is an \textit{auxiliary index}: it does not occur as an exterior
exponent in the displayed bundle, but records the cohomological range. The
analogue of \eqref{eq:weight-one-koszul-resolution} on $C_m$ shows that
$(\mathbf {III}_m)$ implies $(\mathbf {II}_m)$. 

The initial tuple $(n,p,\boldsymbol0,0)$ is admissible. For this tuple,
$(\mathbf I_n)$ and $(\mathbf {III}_n)$ are precisely $(\mathrm A)$ and
$(\mathrm B)$, respectively.

Before formulating the induction, we examine the terminal coefficients left by
taking the kernel term at every remaining $\Gamma$-factor. Let the
point twists be indexed by
$J:=\{\ell:a_\ell\neq0\}$. 

For every admissible tuple $(m,p_0,\boldsymbol c,k)$, In $(\mathbf I_m)$ one has
$|J|\leq p_0+1$, and \eqref{eq:weight-one-B-degree} makes the terminal
coefficient nonspecial, since
\[
\deg\left(
B_{\boldsymbol c,k}-\sum_{\ell\in J}P_\ell
\right)\geq(m+1)g-1\geq 2g-1.
\]
In $(\mathbf {III}_m)$ one has $|J|\leq p_0$ and
$1\leq i\leq m$, and the coefficient is
\[
M_i:=B_{\boldsymbol c,k}-\sum_{\ell\in J}P_\ell-iA.
\]
By \eqref{eq:weight-one-B-degree}, we have
\begin{equation}\label{eq:weight-one-M-degree}
\deg(M_i)\geq(n+1-k-i)g+e-i.
\end{equation}
Since $k\leq n-m$ and $i\leq m$, we have $k+i\leq n$. If $k+i<n$, then
\[
\deg(K_C-M_i)\leq i-e-2\leq i-2,
\]
so the bound $h^0(N)\leq\max\{0,\deg N+1\}$ and Serre duality give
$h^1(C,M_i)\leq i-1$. The tautological cohomology
formula then gives $H^i(C_m,T_{m,M_i})=0$.

When $k+i=n$, necessarily $k=n-m$ and $i=m$. Since $A=L-\sum P_\ell$, the coefficient becomes
\begin{equation}\label{eq:weight-one-weighted-M}
M_m=
B-nL+
\sum_{\ell=0}^{r-2}
\left(n-c_\ell-\mathbf1_{\{\ell\in J\}}\right)P_\ell.
\end{equation}
The multiplicities are at least $m-1\geq0$, while
$\deg M_m\geq g+e-m$. The required vanishing is
\[
H^m(C_m,T_{m,M_m})=\bigwedge^mH^1(C,M_m)=0.
\]
For $g=0$, the degree estimate gives $h^1(C,M_m)\leq m-1$ directly.
For $g\geq1$, by Serre duality,
$$H^1(C,M_m)^{\vee}\cong H^0\left(K_C-(B-nL)-\sum_{\ell=0}^{r-2}
\left(n-c_\ell-\mathbf1_{\{\ell\in J\}}\right)P_\ell\right)$$
If the vanishing orders at the flag points impose independent
conditions, Riemann--Roch gives
\begin{align*}
h^1(C,M_m) &= h^1(C,B-nL)-\sum_{\ell=0}^{r-2}
\left(n-c_\ell-\mathbf1_{\{\ell\in J\}}\right)\\ 
&= h^0(C,B-nL)-\deg M_m+g-1.
\end{align*}
Thus, for sufficiently general flag points $\{P_{\ell}\}_{\ell=0}^{r-2}$, as $\deg M_m\geq g+e-m$,
\[
h^1(C,M_m)=\max\{0,h^0(C,B-nL)-\deg M_m+g-1\}
\leq\max\{0,h^0(C,B-nL)-e+m-1\}.
\]

If $h^0(C,B-nL)\leq e$, then $h^1(C,M_m)\leq m-1$, hence $H^m(C_m,T_{m,M_m})=0$.
 
If
$H^1(C,B-nL)=0$, the effectiveness of $M_m-(B-nL)$ implies that
$h^1(C,M_m)=0$. Hence $H^m(C_m,T_{m,M_m})=0$.
It leads to the auxiliary condition
\begin{equation*}\tag{H3}\label{hyp:weight-one-residual}
g=0,\quad\text{or}\quad h^0(C,B-nL)\leq e,\quad\text{or}\quad H^1(C,B-nL)=0.
\end{equation*}
There are only finitely many admissible data and multiplicity patterns. Hence the above argument proves the following
Lemma~\ref{lem:weight-one-coefficients} under
this additional assumption. 

\begin{lemma}\label{lem:weight-one-coefficients}Assume conditions~(\ref{hyp:weight-one-B}),
(\ref{hyp:weight-one-L}), and~(\ref{hyp:weight-one-residual}). The flag points $\{P_{\ell}\}_{\ell=0}^{r-2}$ may be chosen generically so that, for every admissible
$(m,p_0,\boldsymbol c,k)$, the following hold.
\begin{enumerate}[\normalfont(i)]
    \item If $J\subseteq\{0,\ldots,r-2\}$ and $|J|\leq p_0+1$, then
    \[
    H^1\left(
    C_m,
    T_{m,B_{\boldsymbol c,k}-\sum_{\ell\in J}P_\ell}
    \right)=0.
    \]
    \item If $J\subseteq\{0,\ldots,r-2\}$, $|J|\leq p_0$, and
    $1\leq i\leq m$, then
    \[
    H^i\left(
    C_m,
    T_{m,B_{\boldsymbol c,k}
    -\sum_{\ell\in J}P_\ell-iA}
    \right)=0.
    \]
\end{enumerate}
\end{lemma}

We first
prove our main theorem under condition~(\ref{hyp:weight-one-residual}); we treat
the remaining cases at the end of the proof. The preceding reduction
shows that Theorem~\ref{thm_weight_one} follows from the following proposition.
\begin{proposition}\label{prop:weight-one-induction}
Assume condition~(\ref{hyp:weight-one-residual}). For every admissible
$(m,p_0,\boldsymbol c,k)$, the assertions
$(\mathbf I_m)$, $(\mathbf {II}_m)$, and $(\mathbf {III}_m)$ hold.
\end{proposition}

We prove the proposition after one technical lemma. It describes how
quotient terms produce admissible data on lower dimension.

\begin{lemma}\label{lem:weight-one-restriction}
Let $m\geq2$, and fix admissible data
$(m,p_0,\boldsymbol c,k)$.
\begin{enumerate}[\normalfont(i)]
    \item Every non-zero graded term in the natural filtration of the
    quotient term arising from $(\mathbf I_m)$ is an instance of
    $(\mathbf I_{m-1})$ or $(\mathbf {II}_{m-1})$ for admissible data on
    dimension $m-1$.
    \item Every such non-zero graded term arising from $(\mathbf {III}_m)$ is, after
    resolving the remaining exterior power of $M_{T_{m-1,A}}$, an instance
    of $(\mathbf {III}_{m-1})$ for admissible data on dimension $m-1$.
\end{enumerate}
\end{lemma}

\begin{proof}
We may discard any graded piece that vanishes for rank reasons. Both
$(\mathbf I_m)$ and $(\mathbf {III}_m)$ contain the same product of
$\Gamma$-factors on $C_m$:
\begin{equation}\label{eqn_gamma_vector_bundle_m}
\bigotimes_{\ell=0}^{r-2}
\bigwedge^{a_\ell}\Gamma^{(m)}_{L_\ell},
\quad\textrm{with } a+\sum_{\ell=0}^{r-2}a_\ell=p_0+1.
\end{equation}
Here $a=0$ for $(\mathbf I_m)$, whereas $a\geq1$ is the auxiliary index for $M_{T_{m,A}}$ in $(\mathbf {III}_m)$. Treat the nonzero $\Gamma_{L_{\ell}}$-factors in decreasing
order. For each one,
\eqref{eqn-wedge-gamma-middle} gives
\[
0\longrightarrow
\bigwedge^{a_\ell}S^{m-1}H^0(L_\ell)
\otimes T_{m,\sO_C(-P_{\ell})}
\longrightarrow
\bigwedge^{a_\ell}\Gamma^{(m)}_{L_\ell}
\longrightarrow
\bigwedge^{a_\ell}M_{T_{m-1,L_\ell}}
\longrightarrow0.
\]
Fix a quotient term, and let $\bar\ell$ be the index at which the quotient
is first chosen. The larger indices have already contributed kernel terms,
and hence the point twist
\[
D_t:=
\sum_{\substack{\ell>\bar\ell, a_\ell\neq0}}P_{\ell},
\qquad
t:=\deg(D_t).
\]
The quotient is supported on $X_{P_{\bar\ell}}\cong C_{m-1}$. The same flag points
$P_{\bar\ell},\ldots,P_{r-2}$ give a filtration of
$M_{T_{m-1,L_{\bar\ell}}}$ whose graded pieces are
$\Gamma^{(m-1)}_{L_\ell}$, for $\bar\ell\leq\ell\leq r-2$, and
$M_{T_{m-1,A}}$. By Remark~\ref{rmk:01}, the graded pieces of
$\bigwedge^{a_{\bar\ell}}M_{T_{m-1,L_{\bar\ell}}}$ have the form
\[
\bigotimes_{\ell=\bar\ell}^{r-2}
\bigwedge^{a_\ell'}\Gamma^{(m-1)}_{L_\ell}
\otimes\bigwedge^{a'}M_{T_{m-1,A}},
\quad
a'+\sum_{\ell=\bar\ell}^{r-2}a_\ell'=a_{\bar\ell}.
\]
For the remaining $\Gamma_{L_\ell}$-factors, where $\ell<\bar\ell$ and hence
$P_{\ell}\neq P_{\bar\ell}$, equation~\eqref{eqn-gamma-rest-middle} gives
\[
0\longrightarrow
S^{m-1}H^0\big(L_\ell(-P_{\bar\ell})\big)\otimes\sO_{C_{m-1}}
\longrightarrow
\Gamma^{(m)}_{L_\ell}\big|_{X_{P_{\bar\ell}}}
\longrightarrow
\Gamma^{(m-1)}_{L_\ell}
\longrightarrow0.
\]
By Lemma~\ref{lem-vanishing-wedge}, the exterior-power filtration of this
restriction sequence has graded pieces
\[
\bigwedge^{a_\ell-a_\ell'}
S^{m-1}H^0\big(L_\ell(-P_{\bar\ell})\big)
\otimes
\bigwedge^{a_\ell'}\Gamma^{(m-1)}_{L_\ell},
\quad 0\leq a_\ell'\leq a_\ell.
\]
Thus, up to vector-space factors, any bundle passed to $C_{m-1}$ from
\eqref{eqn_gamma_vector_bundle_m} on $C_m$ has the form:
\begin{equation*}
\bigotimes_{\ell=0}^{r-2}
\bigwedge^{a_\ell'}\Gamma^{(m-1)}_{L_\ell}
\otimes\bigwedge^{a'}M_{T_{m-1,A}}, \quad
a'+\sum_{\ell=\bar\ell}^{r-2}a_\ell'=a_{\bar\ell},
\quad
0\leq a_\ell'\leq a_\ell\quad(\ell<\bar\ell).
\end{equation*}

First consider $(\mathbf I_m)$, a bundle $\displaystyle\bigotimes_{\ell=0}^{r-2}\bigwedge^{a_\ell}\Gamma^{(m)}_{L_\ell}\otimes T_{m,B_{\boldsymbol c,k}}$ on $C_m$, after the above quotient reduction to $X_{P_{\bar{\ell}}}\cong C_{m-1}$, has the form:
\[
\bigotimes_{\ell=0}^{r-2}
\bigwedge^{a_\ell'}\Gamma^{(m-1)}_{L_\ell}
\otimes\bigwedge^{a'}M_{T_{m-1,A}}\otimes T_{m-1,B_{\boldsymbol c',k}},\quad a'+\sum_{\ell=\bar\ell}^{r-2}a_\ell'=a_{\bar\ell},
\quad
0\leq a_\ell'\leq a_\ell\quad(\ell<\bar\ell).
\]
In the above, write $D_t=\sum_{\ell=0}^{r-2}\tau_\ell P_\ell$, where
$\tau_\ell\in\{0,1\}$, then
$\boldsymbol c'=\boldsymbol c+\boldsymbol\tau$. Then
\[
B_{\boldsymbol c',k}=B_{\boldsymbol c,k}-D_t,\quad c'_\ell\leq n-(m-1),
\quad k\leq n-(m-1).
\]
Let $p_0'+1:=a'+\sum_{\ell=0}^{r-2}a_\ell'$. Since $a_\ell\geq1$ for each of the $t$ indices contributing to $D_t$, we have $1\leq p_0'+1\leq\sum_{\ell\leq\bar\ell}a_\ell \leq p_0+1-t$, equivalently, $0\leq p_0'\leq p_0-t$. Therefore,
\[
|\boldsymbol c'|+k+p_0'
\leq|\boldsymbol c|+t+k+p_0-t\leq p.
\]
Thus the data $(m-1,p_0',\boldsymbol c',k)$ are admissible. If $a'=0$, it is an instance of
$(\mathbf I_{m-1})$; if $1\leq a'\leq m-1$, it is an instance of
$(\mathbf {II}_{m-1})$. If $a'\geq m$, it vanishes for rank reasons. This proves (i).

Now we consider $(\mathbf {III}_m)$ where the auxiliary index $a\geq1$, hence
$\sum_{\ell=0}^{r-2}a_\ell\leq p_0$. A bundle $\displaystyle \bigotimes_{\ell=0}^{r-2} \bigwedge^{a_\ell}\Gamma^{(m)}_{L_\ell} \otimes T_{m,B_{\boldsymbol c,k}-iA}$ on $C_m$, after the above quotient reduction to $X_{P_{\bar{\ell}}}\cong C_{m-1}$, has the forms: 
\[
    \bigotimes_{\ell=0}^{r-2}
\bigwedge^{a_\ell'}\Gamma^{(m-1)}_{L_\ell}
\otimes\bigwedge^{a'}M_{T_{m-1,A}}\otimes T_{m-1,B_{\boldsymbol c',k}-iA},\quad a'+\sum_{\ell=\bar\ell}^{r-2}a_\ell'=a_{\bar\ell},
\quad
0\leq a_\ell'\leq a_\ell\quad(\ell<\bar\ell).
    \]
Set $$\beta:=\sum_{\ell=0}^{r-2}a_\ell'.$$ Let the new auxiliary index $\alpha:=1$ for both of the following two cases. Hence
$\alpha+\sum_{\ell=0}^{r-2}a_\ell'=\beta+1$,
which satisfies the index relation for $(\mathbf {III}_{m-1})$.

Suppose first that $a'=0$. Then
$\beta\leq\sum_{\ell\leq\bar\ell}a_\ell\leq p_0-t$, since
$\sum a_\ell\leq p_0$ and each of the $t$ indices contributing to $D_t$
has $a_\ell\geq1$.
Together with the coordinate bounds verified above, the inequality
\[
|\boldsymbol c'|+k+\beta
\leq|\boldsymbol c|+t+k+p_0-t\leq p
\]
shows that $(m-1,\beta,\boldsymbol c',k)$ is admissible. If $i\leq m-1$,
the resulting group is $(\mathbf {III}_{m-1})$ for these data with
auxiliary index $\alpha=1$; if $i>m-1$, it vanishes by dimension.

Suppose now that $a'>0$. We may assume $1\leq a'\leq m-1$, for rank reasons. Apply
\eqref{eq:weight-one-koszul-resolution} with $(n,a)$ replaced by
$(m-1,a')$. For each $1\leq j\leq m-a'$, it suffices to prove the
vanishing of
\begin{equation}\label{eq:weight-one-restriction-cohomology}
H^{i+j-1}\left(
C_{m-1},
\bigotimes_{\ell=0}^{r-2}
\bigwedge^{a_\ell'}\Gamma^{(m-1)}_{L_\ell}
\otimes T_{m-1,B_{\boldsymbol c',k}-(i+j)A}
\right),
\end{equation}
Set $h:=i+j-1$. Since $a,a'\geq1$,
$\beta\leq\sum_{\ell\leq\bar\ell}a_\ell-a'\leq p_0-t-1$. Consequently, $(m-1,\beta,\boldsymbol c',k+1)$ is admissible: indeed,
\[
c'_\ell,k+1\leq n-(m-1),
\quad
|\boldsymbol c'|+(k+1)+\beta
\leq|\boldsymbol c|+t+k+1+p_0-t-1\leq p.
\]
Finally,
\[
B_{\boldsymbol c',k}-(i+j)A
=B_{\boldsymbol c',k+1}-hA.
\]
Therefore, if $h\leq m-1$, the group in
\eqref{eq:weight-one-restriction-cohomology} is an instance of
$(\mathbf {III}_{m-1})$ for the data
$(m-1,\beta,\boldsymbol c',k+1)$ with auxiliary index $\alpha=1$; if
$h>m-1$, it vanishes by dimension. This proves (ii).
\end{proof}

Choose the flag points as in Lemma~\ref{lem:weight-one-coefficients} so that all terminal coefficients have the required
cohomology vanishing.

\begin{proof}[Proof of Proposition~\ref{prop:weight-one-induction}]
It is enough to prove $(\mathbf{I}_m)$ and $(\mathbf{III}_m)$.

We argue by induction on $m$, applying the induction hypothesis only to
quotient terms.
For $m=1$, $\Gamma^{(1)}_{L_\ell}\cong\sO_C(-P_{\ell})$.
For nonzero exterior powers, $(\mathbf I_1)$ is
Lemma~\ref{lem:weight-one-coefficients}~(i). In
$(\mathbf {III}_1)$ one necessarily has $a=i=1$, and the assertion is
Lemma~\ref{lem:weight-one-coefficients}~(ii).

Assume $m\geq2$ and the proposition holds at level $m-1$. Let $\ell_1>\cdots>\ell_s$
be the indices for which $a_{\ell_j}\neq0$, and let $\mathcal E_0$ denote
the vector bundle occurring in $(\mathbf I_m)$ or $(\mathbf {III}_m)$. Let $i_j:X_{P_{\ell_j}}\cong C_{m-1}\hookrightarrow C_m$ be the inclusion. Successive application of
\eqref{eqn-wedge-gamma-middle} gives exact sequences
\begin{equation}\label{eq:weight-one-kernel-quotient}
0\longrightarrow\mathcal E_j
\longrightarrow\mathcal E_{j-1}
\longrightarrow i_{j*}\mathcal Q_j
\longrightarrow0,
\qquad 1\leq j\leq s.
\end{equation}
Here $\mathcal E_j$ is obtained by taking the kernel term for
$\ell_1,\ldots,\ell_j$, while $\mathcal Q_j$ is supported on
$X_{P_{\ell_j}}\cong C_{m-1}$.
The intermediate $\mathcal E_j$ stay on $C_m$ and need not have
admissible data. The full set of point twists is
$J=\{\ell:a_\ell\neq0\}$ considered above; the induction hypothesis
will be applied only to the quotients $\mathcal Q_j$ on $C_{m-1}$.

The terminal $\mathcal E_s$ has no $\Gamma$-factor. The required cohomology vanishing follows from
\[
H^1\left(C_m,T_{m,B_{\boldsymbol c,k}-\sum_{\ell\in J}P_\ell}\right)=0
\]
in $(\mathbf I_m)$, and
\[
H^i\left(C_m,T_{m,B_{\boldsymbol c,k}-\sum_{\ell\in J}P_\ell-iA}\right)=0,
\]
in $(\mathbf{III}_m)$.
The first is covered by Lemma~\ref{lem:weight-one-coefficients}(i), since
\[
\#\{\ell:a_\ell\neq0\}=|J|\leq p_0+1.
\]
The second is covered by part (ii), since $a\geq 1$ and
\[
\#\{\ell:a_\ell\neq0\}=|J|
\leq\sum a_\ell=p_0+1-a\leq p_0,
\qquad
i\leq m.
\]

By contrast, each $\mathcal Q_j$ is a genuine restriction term to the divisor $C_{m-1}$.
Lemma~\ref{lem:weight-one-restriction} filters $\mathcal Q_j$ into
admissible terms at level $m-1$. The induction hypothesis, together with
$(\mathbf {III}_{m-1})\Rightarrow(\mathbf {II}_{m-1})$, gives
$H^1(C_{m-1},\mathcal Q_j)=0$ in the proof of $(\mathbf I_m)$ and
$H^i(C_{m-1},\mathcal Q_j)=0$ in the proof of $(\mathbf {III}_m)$.
Writing $u=1$ in the first case and $u=i$ in the second, we have
\[
H^u(C_m,i_*\mathcal Q_j)=H^u(C_{m-1},\mathcal Q_j)= 0,
\qquad 1\leq j\leq s.
\]
Starting with the directly controlled terminal kernel $\mathcal E_s$ and
using the cohomology sequences of
\eqref{eq:weight-one-kernel-quotient} in reverse order, we obtain
successively
\[
H^u(C_m,\mathcal E_{s-1})=\cdots=H^u(C_m,\mathcal E_0)=0.
\]
This proves $(\mathbf I_m)$ and $(\mathbf {III}_m)$; hence
$(\mathbf {II}_m)$ follows.
\end{proof}

\begin{proof}[Proof of Theorem~\ref{thm_weight_one}]
Under condition~(\ref{hyp:weight-one-residual}),
Proposition~\ref{prop:weight-one-induction} and the preceding reduction prove
the theorem. If $n = 1$ and $g\geq 1$, condition~(\ref{hyp:weight-one-L}) gives $\deg (B-L)\leq e-1$, hence
$h^0(C,B-L) \leq e$, so condition~(\ref{hyp:weight-one-residual}) holds in the curve case. It remains to consider $g\geq1$, $n\geq2$, and the failure of
condition~(\ref{hyp:weight-one-residual}). 

Then $h^0(C,B-nL)>e\geq0$ and
$h^1(C,B-nL)>0$, so $B-nL$ is effective and special. 
We have
\[
0\leq \deg (B-nL)\leq2g-2,\qquad h^1(C,B-nL)\leq g.
\]
The last inequality follows by multiplying sections of $K_C-(B-nL)$ by a
nonzero section of $B-nL$.

Put $d:=\deg L\geq 2g+1$, $r=r(L)=d-g\geq g+1$. Since $\deg B\geq (n+1)g+p$, we obtain
\begin{align*}
p \leq \deg B-(n+1)g
&=n\deg L+\deg (B-nL)-(n+1)g\\
&\leq n(r+g)+2g-2-(n+1)g\\ 
&= nr+g-2.
\end{align*}

Denote $\mathcal L:=T_{n,L}$, $\mathcal F:=T_{n,B-nL}$, and
\[
R:=\operatorname{rk}M_{\mathcal L}=\binom{r+n}{n}-1.
\]
Since $\deg(B-nL)\geq0$ and $d\geq2g+1$, every $B-nL+jL$, $j\geq1$, is
nonspecial. Thus $H^i(C_n,\mathcal F+j\mathcal L)=0$ for $i>0$, $j\geq1$. By the coefficient shift and Lemma~\ref{cor:el12-asymptotic-syzygies}, we have
\vspace{0.25em}
\[
\begin{split}
K_{p,1}(C_n,T_{n,B};T_{n,L})
&=K_{p,n+1}(C_n,\mathcal F;\mathcal L)\\
&\cong H^n\left(C_n,\bigwedge^{p+n}M_{\mathcal L}
                \otimes(\mathcal F+\mathcal L)\right).
\end{split}
\]
Serre
duality and $\det M_{\mathcal L}=\mathcal L^{-1}$ give the following identification:
\begin{equation}\label{eq:weight-one-top-row}
K_{p,1}(C_n,T_{n,B};T_{n,L})
\cong
H^0\left(C_n,\bigwedge^{R-p-n}M_{\mathcal L}
              \otimes(K_{C_n}-\mathcal F)\right)^{\vee}.
\end{equation}
By Lemma~\ref{lem:green-vanishing}, the desired vanishing follows
if $R-p-n\geq h^0(C_n,K_{C_n}-\mathcal{F})$. Serre duality and the
tautological cohomology formula give
\[
h^0(C_n,K_{C_n}-\mathcal{F})
=h^n(C_n,T_{n,B-nL})=\binom{h^1(C,B-nL)}{n}\leq\binom{g}{n}.
\]

It therefore suffices to prove $R-p-n\geq\binom{g}{n}$.
Recall that $n\geq 2$, $g\geq 1$, and
\begin{equation}\label{align_p}
r-1\geq g,\qquad p\leq nr+g-2,\qquad
R=\binom{r+n}{n}-1.
\end{equation}
By Vandermonde's identity, retaining the terms indexed by $0,1,n$,
we obtain
\[
\begin{aligned}
\binom{r+n}{n}
&=\sum_{j=0}^{n}\binom{r-1}{j}\binom{n+1}{n-j}\\
&\geq \binom{r-1}{n}+(r-1)\binom{n+1}{2}+n+1.
\end{aligned}
\]
Consequently, using the relation~\eqref{align_p}, we have
\[
\begin{aligned}
R-p-n
&\geq \binom{r-1}{n}+(r-1)\binom{n+1}{2}-nr-g+2=\binom{r-1}{n}+(r-1)\binom{n}{2}-n-g+2\\
&\geq \binom{g}{n}+g(n-1)-n-g+2=\binom{g}{n}+(n-2)(g-1)\geq \binom{g}{n},
\end{aligned}
\]
where the second inequality follows from $r-1\geq g$ and
$\binom{n}{2}\geq n-1$.
This completes the proof.
\end{proof}

\section{Further questions}\label{sec-questions}
\subsection{Gonality conjecture for symmetric products of curves}
The gonality theorem for curves of genus $g\geq2$ describes the weight-one syzygies of a sufficiently positive line bundle. If
\[
\deg(L)\geq2g+\gon(C),
\]
then
\[
K_{p,1}(C,L)\neq0\quad\Longleftrightarrow\quad1\leq p\leq r(L)-\gon(C),
\]
where $r(L)=h^0(C,L)-1$. Green and Green--Lazarsfeld proved the nonvanishing direction~\cite{Green:KoszulI,Lazarsfeld:SyzFiniteSet}; Ein--Lazarsfeld proved asymptotic vanishing~\cite{Ein:Gonality}; and Niu--Park established the effective bound above~\cite{Niu:Gonality}.

To study the analogous question on $C_n$, assume $\deg(L)\geq2g+\gon(C)$ and put $d=\deg(L)$ and $r=r(L)=d-g$. Following~\cite{Park:GonSeql}, write
\[
d_j(C):=\min\{\deg(A):h^0(C,A)\geq j+1\},\qquad
\gamma_j(C):=d_j(C)-j\quad(j\geq0).
\]
Thus $(d_j(C))_{j\geq0}$ is the gonality sequence, with $d_1=\gon(C)$. By minimality, a bundle $A_j$ attaining $d_j$ is globally generated and has exactly $j+1$ sections. If $j\geq1$ and $d_j\leq d-2g$, then $h^0(C,L-A_j)=r-d_j+1\geq g+1$. Applying Green--Lazarsfeld nonvanishing~\cite[Appendix]{Green:KoszulI} to
\[
T_{n,L}=T_{n,A_j}+T_{n,L-A_j}
\]
and using the tautological cohomology formula gives
\[
K_{P_{n,j}(L),1}(C_n,T_{n,L})\neq0,\qquad
P_{n,j}(L):=\binom{n+j}{n}+\binom{n+r-d_j}{n}-3.
\]
In this degree range, the largest index occurs at $j=1$. Indeed, subtracting $j-1$ points gives $d_j\geq d_1+j-1$, so
\[
j,\ r-d_j\geq1,\qquad j+(r-d_j)\leq r-d_1+1.
\]
Since $t\mapsto\binom{n+t}{n}$ is increasing and convex on nonnegative integers, $P_{n,j}(L)\leq P_{n,1}(L)$. Thus
\begin{equation}\label{eq:gonality-factorization-index}
\Lambda_n(C,L):=\max_{\substack{j\geq1\\d_j\leq d-2g}}P_{n,j}(L)
=\binom{n+r-\gon(C)}{n}+n-2.
\end{equation}

It follows that
\[
K_{p,1}(C_n,T_{n,L})\neq0
\quad\text{for}\quad 1\leq p\leq\Lambda_n(C,L).
\]
We conjecture that there are no further linear syzygies once $L$ is sufficiently positive.

\begin{conjecture}\label{ques:last-linear-syzygy}
For every smooth projective curve $C$ of genus $g\geq1$ and every $n\geq1$, there exists $d_0(C,n)\geq2g+\gon(C)$ such that, for every line bundle $L$ with $\deg(L)\geq d_0(C,n)$,
\[
K_{p,1}(C_n,T_{n,L})\neq0\quad\Longleftrightarrow\quad1\leq p\leq\Lambda_n(C,L).
\]
\end{conjecture}

For $n=1$ this is the gonality theorem; for $n>1$, no effective value of $d_0(C,n)$ is proposed. For $C=\PP^1$ and $d\geq2$, decompositions $d=j+(d-j)$ with $1\leq j\leq d-1$ give the analogous threshold
\[
\Lambda_n(\PP^1,\sO_{\PP^1}(d))=\binom{n+d-1}{n}+n-2.
\]
Under $C_n\cong\PP^n$, the corresponding vanishing assertion is the weight-one, coefficient-zero case of~\cite[Conjecture~7.5]{ein-lazarsfeld-asymptotic}.

The effective vanishing direction for curves follows from the following result~\cite[Theorem~1.3]{Niu:Gonality}.
\begin{theorem}[Niu--Park]\label{thm-niu-park}
Let $B$ and $L$ be line bundles on a smooth projective curve $C$ of genus $g\geq0$. If $B$ is $p$-very ample and
\[
h^1(C,L-B)\leq r(B)-p-1,
\]
then $L$ is $p$-very ample and $K_{p,1}(C,B;L)=0$.
\end{theorem}

For $n=1$ and $g\geq1$, Theorem~\ref{thm_weight_one} assumes $\deg(B)\geq2g+p$ and $\deg(L)\geq2g+p+1$. The first bound makes $B$ nonspecial and $p$-very ample. Riemann--Roch then rewrites the cohomological condition in Theorem~\ref{thm-niu-park} as
\[
\deg(L)\geq2g+p+h^0(C,L-B).
\]
As $h^0(C,L-B)
\le\max\{1,\deg(L-B)\}
\le\deg L-2g-p$.
This recovers the base case of Theorem~\ref{thm_weight_one}.
An analogue of Theorem~\ref{thm-niu-park} for symmetric products might establish the vanishing predicted by Conjecture~\ref{ques:last-linear-syzygy}.
\subsection{Green's conjecture for symmetric product of curves}
For the canonical polarization, higher terms of the gonality sequence can contribute. Assume that $C$ is nonhyperelliptic of genus $g\geq3$. Then $T_{n,K_C}$ is very ample by Proposition~\ref{prop-p-very-ample}; for $n>1$ it differs from $K_{C_n}=T_{n,K_C}(-\delta_n)$. Define the Clifford sequence by
\[
c_j(C):=\min\{\deg(A)-2j:h^0(C,A)=j+1,\ h^1(C,A)\geq2\}\quad(j\geq1),
\]
with $c_j(C)=+\infty$ if the set is empty. Since $h^1(C,A_j)=g-\gamma_j$, $c_j(C)$ is finite precisely when $\gamma_j\leq g-2$, and then $c_j=d_j-2j=\gamma_j-j$.

For any $A$ with $a=h^0(A)\geq2$ and $b=h^1(A)\geq2$, Green--Lazarsfeld applied to $K_C=A+(K_C-A)$ gives directly
\[
K_{\binom{n+a-1}{n}+\binom{n+b-1}{n}-3,1}(C_n,T_{n,K_C})\neq0.
\]
For fixed $a=j+1$, the largest $b$ is $g-j-c_j$. Thus the largest index supplied by these decompositions is
\begin{equation}\label{eq:canonical-factorization-index}
\Lambda_n^{\mathrm{can}}(C):=\max_{\substack{j\geq1\\c_j(C)<+\infty}}
\left\{\binom{n+j}{n}+\binom{n+g-1-j-c_j(C)}{n}-3\right\},
\end{equation}
where an empty maximum is $0$. This shows that nonvanishing holds for $1\leq p\leq\Lambda_n^{\mathrm{can}}(C)$.

\begin{conjecture}\label{conj:tautological-green}
For every smooth nonhyperelliptic curve $C$ of genus $g\geq3$ and every $n\geq1$,
\[
K_{p,1}(C_n,T_{n,K_C})=0\quad\text{for all}\quad p>\Lambda_n^{\mathrm{can}}(C).
\]
\end{conjecture}

For $n=1$ and $g\geq4$, the equality
$\Lambda_1^{\mathrm{can}}(C)=g-2-\operatorname{Cliff}(C)$ recovers Green's
conjecture. For $n=2$, the bound depends on more than the Clifford index.
The following two genus-six examples have Clifford index one, so their admissible
decompositions satisfy $a+b\leq6$. Convexity bounds the index by $10$,
which the trigonal pencil with $(a,b)=(2,4)$ attains. On a smooth plane
quintic, $\cO_C(1)$ has $(a,b)=(3,3)$ and gives index $9$. Gonality four
excludes $(2,4)$ and $(4,2)$ on the quintic, so $9$ is maximal. Thus the
two curves have the same Clifford index but different values of
$\Lambda_2^{\mathrm{can}}$.

For a general curve of genus $g\geq4$, put $k=\lfloor g/2\rfloor$. Since $\operatorname{Cliff}(C)=\lfloor(g-1)/2\rfloor$, every admissible decomposition satisfies
\[
a,b\geq2,\qquad a+b=g+1-\operatorname{Cliff}(A)\leq k+2.
\]
Convexity shows that the largest index occurs at $(a,b)=(2,k)$, which is attained by a minimal pencil. Hence
\[
\Lambda_n^{\mathrm{can}}(C)=\binom{n+k-1}{n}+n-2.
\]
This gives the following special case of Conjecture~\ref{conj:tautological-green}.

\begin{conjecture}\label{ques:generic-tautological-green}
Fix $g\geq4$ and $n\geq2$, and set $k=\lfloor g/2\rfloor$. For a general smooth curve $C$ of genus $g$, 
\[
K_{p,1}(C_n,T_{n,K_C})=0\quad\text{for all}\quad p\geq\binom{n+k-1}{n}+n-1.
\]
For $n=2$, this predicts $K_{p,1}(C_2,T_{2,K_C})=0$ for all $p\geq\binom{k+1}{2}+1$.
\end{conjecture}

For higher weights, the work of Choe--Kwak--Park suggests that the full gonality sequence may be relevant~\cite{Park:GonSeql}.
\begin{question}
Can the gonality sequence describe the nonvanishing range of the higher-weight Koszul cohomology of $T_{n,L}$?
\end{question}

{\parskip=12pt 

}

\begin{thebibliography}{[EHKSY]}

\bibitem[BL25]{bangere-lacini-adjoint}
Purnaprajna Bangere and Justin Lacini.
\newblock {\em Syzygies of adjoint linear series on projective varieties}.
\newblock Duke Math. J., \textbf{174}(3):473--499, 2025.

\bibitem[CKP23]{Park:GonSeql}
Junho Choe, Sijong Kwak, and Jinhyung Park.
\newblock Syzygies of secant varieties of smooth projective curves and gonality
sequences.
\newblock {\em arXiv:2307.03405}, 2023.

\bibitem[EL93]{ein-lazarsfeld-arbitrary} Lawrence Ein and Robert Lazarsfeld. \textit{Syzygies and Koszul cohomology of smooth projective varieties of arbitrary dimension}. Invent. Math.,
111(1):51–67, 1993.

\bibitem[EL12]{ein-lazarsfeld-asymptotic}
Lawrence Ein and Robert Lazarsfeld.
\newblock {\em Asymptotic syzygies of algebraic varieties}.
\newblock Invent. Math., \textbf{190}:603--646, 2012.

\bibitem[EL15]{Ein:Gonality}
Lawrence Ein and Robert Lazarsfeld,
\newblock {\em The gonality conjecture on syzygies of algebraic curves of large
	degree},
\newblock Publ. Math. Inst. Hautes \'Etudes Sci., \textbf{122}: 301--313, 2015.

\bibitem[EL26]{EL26}
Lawrence Ein and Robert Lazarsfeld.
\emph{Lectures on the Syzygies and Geometry of Algebraic Varieties}.
University Lecture Series, vol.~81.
American Mathematical Society, Providence, RI, 2026.

\bibitem[ELY16]{ELY16}
Lawrence Ein, Robert Lazarsfeld, and David Yang.
\newblock A vanishing theorem for weight-one syzygies.
\newblock \emph{Algebra Number Theory}, 10(9):1965--1981, 2016.

\bibitem[ENP20]{ENP}
Lawrence Ein, Wenbo Niu, and Jinhyung Park,
\newblock {\em Singularities and syzygies of secant varieties of nonsingular
	projective curves},
\newblock Invent. Math., \textbf{222} (2): 615--665, 2020.


\bibitem[Gre84a]{Green:KoszulI}
Mark~L. Green,
\newblock {\em Koszul cohomology and the geometry of projective varieties},
\newblock J. Differential Geom., \textbf{19} (1): 125--171, 1984.

\bibitem[Gre84b]{Green:KoszulII}
Mark~L. Green,
\newblock {\em Koszul cohomology and the geometry of projective varieties. II},
\newblock J. Differential Geom., \textbf{20} (1): 279--289, 1984.

\bibitem[GL88]{Lazarsfeld:SyzFiniteSet}
M.~Green and R.~Lazarsfeld.
\newblock Some results on the syzygies of finite sets and algebraic curves.
\newblock {\em Compositio Math.}, 67(3):301--314, 1988.

\bibitem[GP96]{GP96}
F. J. Gallego and B. P. Purnaprajna,
\newblock {\em Higher syzygies of elliptic ruled surfaces},
\newblock J. Algebra, \textbf{186}(2):626--659, 1996.


\bibitem[Har77]{hartshorne} Robin Hartshorne, {\em{Algebraic Geometry}}, Graduate Texts in Mathematics \textbf{52}, Springer-Verlag, New York, 1977.

\bibitem[Mum70]{Mum70}
David Mumford.
\newblock Varieties defined by quadratic equations.
\newblock In \emph{Questions on Algebraic Varieties}
(C.I.M.E., III Ciclo, Varenna, 1969), pages 29--100.
Edizioni Cremonese, Rome, 1970.

\bibitem[NP24]{Niu:Gonality}
Wenbo Niu and Jinhyung Park.
\newblock Effective gonality theorem on weight-one syzygies of algebraic
curves.
\newblock To appear in {\em Duke Math. J.}; arXiv:2405.13446.

\bibitem[Par05]{Park:RuledSurfacesII}
Euisung Park.
\newblock {\em On higher syzygies of ruled surfaces II}.
\newblock J. Algebra, \textbf{294}(2):590--608, 2005.


\bibitem[Rei88]{reider-surface}
Igor Reider. {\em{Vector bundles of rank 2 and linear systems on algebraic
surfaces}}. Ann.\ of Math. (2), 127(2):309–316, 1988.

\bibitem[She20]{She20}
John Sheridan.
\newblock \emph{Divisor varieties and syzygies of symmetric
products of curves}.
\newblock Ph.D. thesis, Stony Brook University, 2020.

\bibitem[She22]{She22}
John Sheridan.
\newblock \emph{Projective normality of canonical symmetric squares}.
\newblock arXiv:2211.08313, 2022.

\end{thebibliography}
\end{document}